\documentclass[12pt]{amsart}

\usepackage{amsthm}
\usepackage{amssymb}

\usepackage[latin1]{inputenc}
\usepackage{subfigure}
\usepackage{color}
\usepackage{amsmath}
\usepackage{amsthm}
\usepackage{amstext}
\usepackage{amssymb}
\usepackage{amsfonts}
\usepackage{graphicx}
\usepackage{young}
\usepackage{multicol}
\usepackage{mathrsfs}
\usepackage{stmaryrd}
\usepackage{bbm}
\usepackage[all]{xy}
\usepackage{mathabx}
\usepackage{tikz}
\usepackage{mathtools}
\usepackage{tabularx}
\usepackage{array}
\usepackage{commath}

\usepackage{caption}

\theoremstyle{plain}
\theoremstyle{definition}
\newtheorem{theorem}{Theorem}[section]

\newtheorem{remark}[theorem]{Remark}
\newtheorem{lemma}[theorem]{Lemma}

\newtheorem{problem}[theorem]{Problem}

\newtheorem{example}[theorem]{Example}
\newtheorem{proposition}[theorem]{Proposition}
\newtheorem{corollary}[theorem]{Corollary}

\DeclareMathAlphabet{\mathpzc}{OT1}{pzc}{m}{it}

\usepackage{amssymb,amsmath,tabularx,graphicx}
\usepackage{tikz}
\usetikzlibrary[calc,intersections,through,backgrounds,arrows,decorations.pathmorphing]
\title{Semistable subcategories for tiling algebras}
\author{Monica Garcia}
\author{Alexander Garver}

\begin{document}

\maketitle

\begin{abstract}{Semistable subcategories were introduced in the context of Mumford's GIT and interpreted by King in terms of representation theory of finite dimensional algebras. Ingalls  and Thomas later showed that for finite dimensional algebras of Dynkin and affine type, the poset of semistable subcategories is isomorphic to the corresponding poset of noncrossing partitions. We show that semistable subcategories defined by tiling algebras, introduced by Coelho Sim{\~o}es and Parsons, are in bijection with noncrossing tree partitions, introduced by the second author and McConville. Moreover, this bijection defines an isomorphism of the posets on these objects. Our work recovers that of Ingalls and Thomas in Dynkin type $A$.}\end{abstract}

\section{Introduction}
Mumford's geometric invariant theory (GIT) provides a technique for taking the quotient of an algebraic variety by certain types of group actions in such a way that the resulting quotient is again an algebraic variety. Given a variety $V$ and a reductive algebraic group $G$ acting linearly on $V$, one replaces $V$ by its ``semistable points'' and then forms the GIT quotient $V\!/\!\!/\! G$, which is an algebraic variety. 

In \cite{king1994moduli}, King interpreted this notion of semistable points in terms of representation theory of algebras as follows. Let $\Lambda = \Bbbk Q/I$ be the path algebra of a \textbf{quiver} $Q$ (i.e., a 4-tuple $(Q_0,Q_1,\textsf{s},\textsf{t})$ where $Q_0  = \{1,2,\ldots, n\}$ is a set of \textbf{vertices}, $Q_1$ is a set of \textbf{arrows}, and two functions $\textsf{s},\textsf{t}: Q_1 \to Q_0$ defined so that for every $\alpha \in Q_1$, we have $\textsf{s}(\alpha) \stackrel{\alpha}{\to} \textsf{t}(\alpha)$) modulo an admissible ideal $I$ and $\Bbbk$ is an algebraically closed field. Recall that the \textbf{path algebra} $\Bbbk Q$ consists of formal $\Bbbk$-linear combinations of paths in $Q$, and its multiplication is induced  by concatenation of paths. For such algebras, any $\Lambda$-module $M$ may be regarded as a \textbf{representation} of $Q$ (i.e., an assignment of a finite dimensional $\Bbbk$-vector space $M_i$ to each vertex of $Q$ and a $\Bbbk$-linear map to each arrow of $Q$). A representation $M$ of $Q$ naturally defines a \textbf{dimension vector}, denoted by $\textbf{dim}(M) := (\dim_\Bbbk M_i)_{i=1}^n \in \mathbb{Z}^n_{\ge 0},$ where $n$ will henceforth denote the number of vertices of $Q$. 

Now let $V=\text{mod}(\Lambda,\textbf{d})$, the variety of finitely generated $\Lambda$-modules with dimension vector $\textbf{d} = (d_1, \ldots, d_n),$  and let $G = \prod_{i=1}^n \text{GL}_{d_i}(\Bbbk)$ act by base change at each vertex of $Q$. In \cite{king1994moduli}, King showed that the semistable points of $V$, which from now on we call \textbf{$\theta$-semistable representations}  (resp., \textbf{$\theta$-stable representations}) where $\theta \in \text{Hom}(\mathbb{Z}^n,\mathbb{Z})$ is a linear map, are the representations $M$ satisfying 

$\bullet$ $\theta(\textbf{dim}(M)) = 0$, and

$\bullet$ for any subrepresentation $N \subset M$, one has $\theta(\textbf{dim}(N)) \le 0$ (resp., $\theta(\textbf{dim}(N)) < 0$).

\noindent We refer to such linear maps $\theta \in \text{Hom}(\mathbb{Z}^n,\mathbb{Z})$ as \textbf{stability conditions} on $\text{mod}(\Lambda)$, the category of finitely generated $\Lambda$-modules. Any choice of stability condition $\theta$ defines a subcategory $\theta^{ss}$ of $\text{mod}(\Lambda)$ consisting of the $\theta$-semistable representations. We refer to $\theta^{ss}$ as a \textbf{semistable subcategory}. Note that two different stability conditions may define the same semistable subcategory.

We study the poset of all semistable subcategories of $\text{mod}(\Lambda)$ ordered by inclusion, denoted $\Lambda^{ss}$. There are close connections between the theory of semistable subcategories and the combinatorics of Coxeter groups. If $\Lambda = \Bbbk Q$ where $Q$ is an acyclic orientation of a simply-laced Dynkin or extended Dynkin diagram, it follows from \cite[Theorem 1.1]{ingalls2009noncrossing} that $\Lambda^{ss}$ is isomorphic to the poset of noncrossing partitions associated with $Q$.

Other important examples of algebras $\Lambda$ include cluster-tilted algebras \cite{bakke2007cluster}, which appear in the context of cluster algebras, and also preprojective algebras. In the latter case, in \cite{thomas2017stability} it is shown that $\Lambda^{ss}$ is isomorphic to the shard intersection order of the Coxeter arrangement associated with $Q$ (see \cite{reading2011noncrossing} for more on the shard intersection order).

The purpose of this work is to combinatorially classify the semistable subcategories for the class of \textbf{tiling algebras}, introduced in \cite{simoes2017endomorphism} to study endomorphism algebras of maximal rigid objects in some negative Calabi-Yau categories. Following \cite{garver_oriented_rep_thy}, these algebras, denoted $\Lambda_T$, are defined by the data of a tree $T$ embedded in the disk $D^2$ whose interior vertices have degree at least 3 (see Figure~\ref{tree_quiver}). Examples of tiling algebras are given by the cluster-tilted algebras of cluster type $A$; the trees defining these algebras are those whose interior vertices are of degree 3.

\begin{figure}
$$\begin{array}{ccccccccc}\includegraphics[scale=1.3]{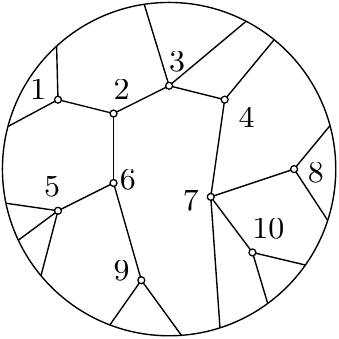} & & & \includegraphics[scale=1.3]{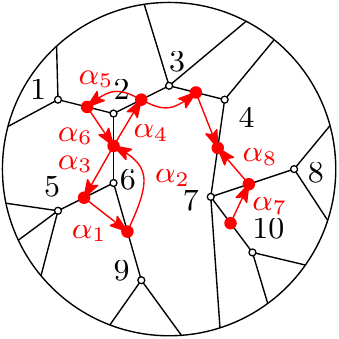} \\ (a) & & & (b)
\end{array}$$
\caption{We show a tree $T$ in (a) and the quiver $Q_T$ it defines in (b). The associated tiling algebra is $\Lambda_T = \Bbbk Q_T/I_I$ where $I_T = \langle \alpha_2\alpha_1, \alpha_3\alpha_2, \alpha_1\alpha_3, \alpha_5\alpha_4, \alpha_6\alpha_5, \alpha_4\alpha_6, \alpha_8\alpha_7 \rangle$.}
\label{tree_quiver}
\end{figure}

The tree $T$ defines a simplicial complex of noncrossing arcs on $T$ called the noncrossing complex, denoted by $\Delta^{NC}(T)$ (see Section~\ref{Sec_nc_complex}). Each facet of $\Delta^{NC}(T)$ consists of \textbf{red arcs}, \textbf{green arcs}, and \textbf{boundary arcs}. In \cite{manneville2017geometric}, it is shown that if $\delta$ is a green or red arc in a facet of $\Delta^{NC}(T)$, it gives rise to a \textbf{g-vector}, denoted $\textbf{g}(\delta) \in \mathbb{Z}^n$. Additionally, in \cite{garver_oriented_rep_thy}, it is shown that the facets of $\Delta^{NC}(T)$ are in bijection with \textbf{wide subcategories} of $\text{mod}(\Lambda_T).$ With these facts in mind, we arrive at our main theorem.

\begin{theorem}\label{thm:thm1}
Let $\mathcal{W} \subset \text{mod}(\Lambda_T)$ be a wide subcategory, let $\mathcal{F}$ be the corresponding facet of ${\Delta}^{NC}(T)$, and let $\mathcal{F}^{gr}$ be the set of green arcs of $\mathcal{F}$. Then the \textbf{Kreweras stability condition} defined as
$$\begin{array}{rcl}
\theta_{\mathcal{F}} : \mathbb{Z}^{n} & \longrightarrow &\mathbb{Z} \\
\textbf{dim}(M) & \longmapsto & \sum_{\delta \in \mathcal{F}^{gr}} \langle \mathbf{g}(\delta), \textbf{dim}(M) \rangle, 
\end{array}$$
where $M \in \text{mod}(\Lambda_T)$ and $\langle -,-\rangle$ is the standard Euclidean inner product, satisfies $\theta^{ss}_{\mathcal{F}} = \mathcal{W}$. Conversely, any semistable subcategory of $\text{mod}(\Lambda_T)$ is a wide subcategory of $\text{mod}(\Lambda_T)$.
\end{theorem}

For simplicity of notation, in the sequel we write $\theta_\mathcal{F}(M)$ rather than the more cumbersome $\theta_{\mathcal{F}}(\textbf{dim}(M)).$

The paper is organized as follows. In Section~\ref{Sec_nc_complex}, we review the noncrossing complex of arcs on a tree. In Section~\ref{Section_cg_vec}, we associate $\textbf{g}$- and $\textbf{c}$-vectors to each facet of this complex, which are essential to our construction of semistable subcategories. In Section~\ref{sec_til_alg}, we define the tiling algebras that we will study. In Section~\ref{Sec_noncrossing_tree_part}, we define noncrossing tree partitions, which will classify the semistable subcategories of $\text{mod}(\Lambda_T)$. In Section~\ref{Section_red-green_torsion_pair}, we describe the data of a noncrossing tree partition and its Kreweras complement as a torsion pair in $\text{mod}(\Lambda_T)$. We use this description to prove Theorem~\ref{thm:thm1} in Section~\ref{Sec_proof_of_thm}. Lastly, in Section~\ref{Sec_additiona_stuff}, we propose a natural extension of our work to general gentle algebras.

\section{Noncrossing complex}\label{Sec_nc_complex}

A \textbf{tree} $T=(V_T,E_T)$ is a finite connected acyclic graph. Any tree may be embedded in the disk $D^2$ in such a way that a vertex is on the boundary if and only if it is a leaf.  We will assume that any tree is accompanied by such an embedding in $D^2$. We say two trees $T$ and $T^\prime$ are \textbf{equivalent} if there is an ambient isotopy between the spaces $D^2\backslash T$ and $D^2\backslash T^\prime$. We consider trees up to equivalence. Additionally, we assume that the \textbf{interior vertices} of any tree $T$ (i.e., the nonleaf vertices of $T$) have degree at least 3. 

We say the closure of a connected component of $D^2\backslash T$ is a \textbf{face} of $T$. A \textbf{corner} $(v,F)$ of $T$ is a pair consisting of an interior vertex $v$ of $T$ and a face $F$ of $T$ that contains $v$.

An \textbf{acyclic path} supported by a tree $T$ is a sequence $(v_0, v_1,...,v_t)$ of pairwise distinct vertices of $T$ such that $v_i$ and $v_j$ are adjacent if and only if $j = i\pm 1$.  By convention, the sequence $(v_0,v_1,\ldots, v_t)$ and the sequence $(v_t, v_{t-1},\ldots,v_0)$ define the same acyclic path. We will refer to $v_0$ and $v_t$ as the \textbf{endpoints} of the acyclic path $(v_0,v_1,\ldots, v_t)$.  Since $T$ is acyclic, any acyclic path is determined by its endpoints, and we can therefore write $(v_0,v_1,...,v_t)=[v_0,v_t]$. In addition, we will say that an acyclic path $[v_0,v_t]$ \textbf{contains} an acyclic path $[u_0,u_s]$ if every vertex of $[u_0,u_s]$ is also a vertex of $[v_0,v_t]$.

Given two acyclic paths $[v_0,v_t]$ and $[v_t,v_{t+s}]$ whose only common vertex is $v_t$ and where $[v_0, v_{t+s}]$ is an acyclic path, we define the \textbf{composition} of $[v_0,v_t]$ and $[v_t,v_{t+s}]$ to be $[v_0,v_t]\circ [v_t,v_{t+s}] := [v_0,v_{t+s}].$

An \textbf{arc} $\delta=(v_0, v_1,...,v_t)$ is an acyclic path such that its endpoints are leaves and any two edges $(v_{i-1},v_i)$ and $(v_i,v_{i+1})$ are incident to a common face. We say $\delta$ \textbf{contains a corner} $(v,F)$ if $v = v_i$ for some $i \in \{1,\ldots, t-1\}$ and $(v_{i-1},v_i)$ and $(v_i,v_{i+1})$ are incident to $F$. We also note that $\delta$ divides $D^2$ into two \textbf{regions} composed of disjoint subsets of the faces of $T$. We let $\text{Reg}(\delta,F)$ denote the region defined by $\delta$ which contains face $F$. We say that two arcs $\delta$ and $\delta^\prime$ are \textbf{crossing} if given any regions $\text{R}_\delta$ and $\text{R}_{\delta^\prime}$ defined by $\delta$ and $\delta^\prime$, respectively, then $\text{R}_\delta \not \subset \text{R}_{\delta^\prime}$ or $\text{R}_{\delta^\prime} \not \subset \text{R}_{\delta}.$ Otherwise, we say $\delta$ and $\delta^\prime$ are \textbf{noncrossing}.

Define the \textbf{noncrossing complex} of $T$, denoted $\Delta^{NC}(T)$, to be the abstract simplicial complex of noncrossing arcs of $T$. By \cite[Corollary 3.6]{garver2016oriented}, this is a \textbf{pure complex} (i.e., any two facets have the same cardinality). We will primarily work with the facets of $\Delta^{NC}(T)$. 

Let $\mathcal{F}$ be any facet of $\Delta^{NC}(T)$. The arcs of $\mathcal{F}$ containing a corner $(v,F)$ are linearly ordered: two arcs $\delta, \gamma \in \mathcal{F}$ containing $(v,F)$ satisfy $\delta \le_{(v,F)} \gamma$ if and only if $\text{Reg}(\delta,F) \subset \text{Reg}(\gamma, F).$ That such arcs are linearly ordered follows from the fact that they are pairwise noncrossing. We say that an arc $\delta$ of $\mathcal{F}$ is \textbf{marked} at corner $(v,F)$ if $\delta$ contains $(v,F)$ and is the maximal such arc with respect to $\le_{(v,F)}$. We denote the unique arc of $\mathcal{F}$ that is marked at corner $(v,F)$ by $p(v,F)$. We use $\delta \lessdot_{(v,F)} \gamma$ to indicate that $\delta <_{(v,F)} \gamma$ and there does not exist $\gamma^\prime \in \mathcal{F}\backslash\{\gamma\}$ such that $\delta <_{(v,F)} \gamma^\prime <_{(v,F)} \gamma$ in $\mathcal{F}$. We show an example of the facets of a noncrossing complex in Figure~\ref{nc_complex}.

In \cite[Proposition 3.5]{garver2016oriented}, it is shown that every $\delta \in \mathcal{F}$ is marked at either one or two corners. In the latter case, the two corners at which $\delta$ is marked belong to different regions defined by $\delta$. We refer to the arcs marked at a single corner as \textbf{boundary arcs}, and we denote the set of boundary arcs of $\mathcal{F}$ by $\mathcal{F}^\partial$. Boundary arcs may also be characterized as the arcs $\delta$ of $T$ with the property that there exists a face $F_\delta$ of $T$ such that every corner contained in $\delta$ is of the form $(v,F_\delta)$ for some interior vertex $v$ of $T$. From this it follows that in any facet $\mathcal{F}$ any boundary arc $\delta$ is minimal with respect to the order $\le_{(v,F_\delta)}$ where $v$ is an interior vertex of $\delta$. In particular, if a boundary arc $\delta \in \mathcal{F}$ is marked at a corner, then $\delta$ is the only arc of $\mathcal{F}$ containing that corner.

\begin{figure}
$$\includegraphics[scale=1.1]{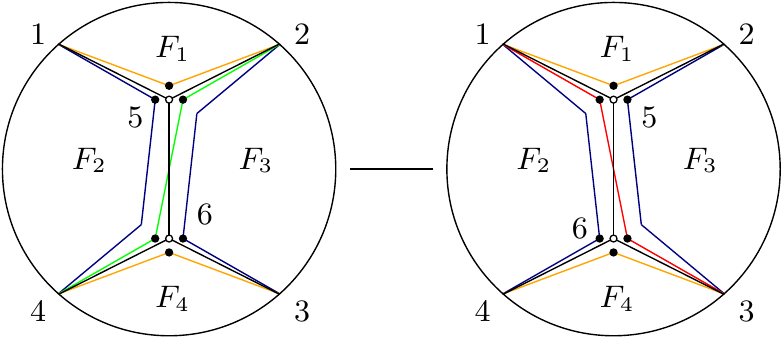}$$
\caption{Both facets of this noncrossing complex contain 5 arcs. Our convention in this paper is to represent an arc $\delta$ from a facet $\mathcal{F}$ as a curve in $D^2$ between the endpoints of $\delta$ that stays close to the vertices in $\delta$, but does not cross any other arcs in $\mathcal{F}$. The boundary arcs are shown in gold and in blue. The marked corners of arcs are indicated by black dots. The faces are $F_1, F_2, F_3, F_4.$ These two facets are joined by an edge to indicate that one facet may be obtained from the other by replacing a red or green arc with an arc of the opposite color.}
\label{nc_complex}
\end{figure}

The arcs of $\mathcal{F}$ that are not boundary arcs come with the extra data of a color as follows. A \textbf{flag} is a triple $(v,e,F)$ consisting of a vertex $v$, an edge $e$, and face $F$ where $v$ is incident to $e$ and $e$ is incident to $F$. We say a flag is \textbf{green} if face $F$ appears immediately counterclockwise from $e$, when rotating about $v$. Otherwise, we say $(v,e,F)$ is \textbf{red}. Let $(v,F)$ and $(u,G)$ be the two corners at which an arc $\delta \in \mathcal{F}\backslash \mathcal{F}^\partial$ is marked, and let $e$ and $e'$ be edges of $T$ contained in $[v,u]$ where the former is incident to $v$ and the latter is incident to $u$. Both $(v,e,F)$ and $(u,e',G)$ have to be of the same color, as $F$ and $G$ belong to different regions determined by $\delta$. We say $\delta$ is a \textbf{green arc} if $(v,e,F)$ and $(u,e',G)$ are green flags, otherwise we say it is a \textbf{red arc}. Define $\mathcal{F}^{gr}$ (resp., $\mathcal{F}^{red}$) to be the set of green (resp., red) arcs of $\mathcal{F}$. Observe that $\mathcal{F}=\mathcal{F}^{red} \sqcup \mathcal{F}^{gr} \sqcup \mathcal{F}^\partial$. We show examples of red and green arcs in Figures~\ref{nc_complex} and \ref{nc_to_ncp}.

Additionally, we can associate to a colored arc $\delta \in \mathcal{F}$ a unique pair of arcs $\{ \mu, \nu \}$ in $\mathcal{F}$. Let $\mu, \nu \in\mathcal{F}$ be the arcs satisfying $\mu \lessdot_{(v,F)} \delta$ and $\nu \lessdot_{(u,G)} \delta$, where $(v,F)$ and $(u,G)$ are the corners at which $\delta$ is marked. If $\delta$ is green (resp., red) we let $F'$ and $G'$ be the unique faces immediately clockwise (resp., counterclockwise) from $F$ about $v$ and $G$ about $u$. By \cite[Proposition 3.7, Claim 1]{garver2016oriented}, arcs $\mu$ and $\nu$ are marked at $(u,G')$ and $(v,F')$ respectively, and $[v,u]$ is the unique longest acyclic path along which they agree. We say that $\mu$ and $\nu$ are the \textbf{supporting arcs} associated to $\delta$ in $\mathcal{F}$. In Figure~\ref{nc_complex}, the supporting arcs of the unique non-boundary arc of each facet are presented in blue. 

We have the following important lemma, which shows that given a non-boundary arc of a facet and one of its supporting arcs, the two have a common leaf of $T$.

\begin{lemma}\label{supp_arcs_lemma}
Let $\mathcal{F} \in \Delta^{NC}(T)$ be a facet, and let $\delta = [u_1,u]\circ[u,v]\circ[v,v_1] \in \mathcal{F}\backslash\mathcal{F}^\partial$ be an arc whose marked corners are $(v,F)$ and $(u,G)$. Let $\{\mu,\nu\}$ be the supporting arcs of $\delta$ where $\mu \lessdot_{(v,F)} \delta$ and $\nu \lessdot_{(u,G)} \delta$. Then $\nu = [u_1,u]\circ [u,v] \circ [v, v^\prime_1]$ and $\mu = [u^\prime_1,u]\circ [u,v] \circ [v,v_1]$ for some acyclic paths $[v,v^\prime_1]$ and $[u^\prime_1,u]$ where $v^\prime_1 \neq v_1$ and $u^\prime_1 \neq u_1$.
\end{lemma}
\begin{proof}

Without loss of generality, we assume that $\delta$ is a green arc. We prove that arc $\nu$ has the desired expression, and the proof that $\mu$ has the desired expression is similar so we omit it.

First, it is clear that $\nu$ and $\delta$ agree along $[u,v]$ and separate at $v$. This means that there exist leaves $u^*_1$ and $v^\prime_1$ of $T$ such that $\nu = [u^*_1,u]\circ [u, v] \circ [v, v_1^\prime]$ and $v^\prime_1 \neq v_1$.

Next, we show that $u^*_1 = u_1$. Suppose that $\delta$ and $\nu$ separate at a vertex $x$ in the acyclic path $[u_1,u]$. Since $\delta$ and $\nu$ both contain $(u,G)$, we know that $x \neq u$. As $\delta$ and $\nu$ are noncrossing and $\delta = p(u,G)$, they must separate as shown in Figure~\ref{delta_gamma_separate}. Let $H$ be the face of $T$ such that corner $(x,H)$ is contained in $\delta$.

Now define $a$ to be the arc of $\mathcal{F}$ that is marked at $(x,H)$. Since $a, \delta$, and $\nu$ are pairwise noncrossing, we have that $[x,v]$ is contained in $a$. It follows that $a$ contains corner $(v,F)$ or $(v,F')$. Suppose the arc $a$ contains corner $(v,F)$, then $a = p(v,F)$ since $\text{Reg}(a,H) = \text{Reg}(a,F)$ and $a = p(x,H)$. However, this implies that $a = p(v,F) = \delta$, a contradiction.  Similarly, if $a$ contains $(v,F^\prime)$, one obtains that $a = p(v,F') = \nu$, a contradiction.

We conclude that there is no vertex $x$ in $[u_1,u]$ at which $\delta$ and $\nu$ separate.
\end{proof}

\begin{figure}
\includegraphics[scale=2]{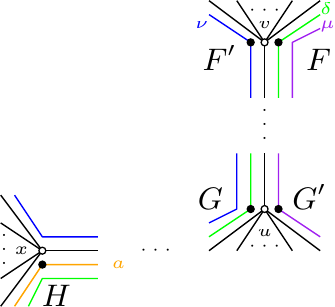}
\caption{The configuration of arcs from the proof of Lemma~\ref{supp_arcs_lemma}.}
\label{delta_gamma_separate}
\end{figure}

\section{Facets and their \textbf{c}- and \textbf{g}-vectors}\label{Section_cg_vec}

In this section, we show how to associate a family of vectors in $\mathbb{Z}^n$ to each facet of the noncrossing complex where $n$ denotes the number of edges of $T$ connecting two interior vertices of $T$. We let $\text{Int}(E_T)$ denote the set of such edges of $T$ and $\{\textbf{x}_e\}_{e \in \text{Int}(E_T)}$ the canonical basis of $\mathbb{Z}^{|\text{Int}(E_T)|} \cong \mathbb{Z}^n$. The definitions we present in this section are reformulations of the definitions presented in \cite{manneville2017geometric}.

Now, fix a facet $\mathcal{F} \in \Delta^{NC}(T)$ and a red or green arc $\gamma = (v_0, v_1 \ldots, v_t) \in \mathcal{F}$. By choosing an orientation of $\gamma$, we define $\textbf{g}(\gamma):= \sum_{e \in \text{Int}(E_T)}g^e_{\gamma}\textbf{x}_e \in \mathbb{Z}^n$ where for each $e = (v_i,v_{i+1}) \in \text{Int}(E_T)$ we set
\[
g^{e}_\gamma := 
\left\{\begin{array}{rll}
1 & \textnormal{if } \gamma \textnormal{ turns left at } v_i \textnormal{ and right at } v_{i+1},\\
-1 & \textnormal{if } \gamma \textnormal{ turns right at } v_i \textnormal{ and left at } v_{i+1}, \\
0  &  \textnormal{if } \gamma \textnormal{ turns in the same direction at } v_{i} \textnormal{ as it}\\ 
 &  \textnormal{does at } v_{i+1} \textnormal{ or if } e \textnormal{ is not an edge in } \gamma, \\
\end{array}\right.
\]
and we refer to $\textbf{g}(\gamma)$ as the $\textbf{g}$\textbf{-vector} of $\gamma$ (see Figure~\ref{g_vec_ex}). Observe that $\textbf{g}(\gamma)$ is independent of the choice of orientation of $\gamma$. We define the \textbf{zigzag} of $\gamma$ to be the set $Z_\gamma = Z_\gamma^+ \sqcup Z_\gamma^- \subset \text{Int}(E_T)$ of edges $e$ of $T$ such that $g_\gamma^e \neq 0$, where $Z_\gamma^+$ (resp., $Z_\gamma^-$) is the set of edges $e$ such that $g_\gamma^e = 1$ (resp., $g_\gamma^e = -1$). We also let $G(\mathcal{F}):= \{\textbf{g}(\gamma)\}_{\gamma\in \mathcal{F}^{red}\sqcup \mathcal{F}^{gr}}$.


\begin{figure}[h]
\includegraphics[width=.70\textwidth]{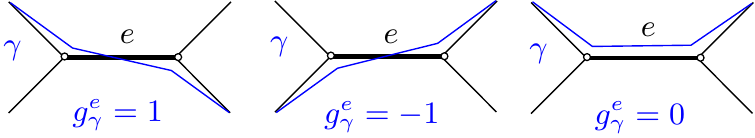}
\centering
\caption{\footnotesize{Different values of $g^e_\gamma$}}
\label{g_vec_ex}
\end{figure}

Next, we let $s_{\gamma,\mathcal{F}} = [v,u]$ denote the acyclic path where $(u,F)$ and $(v,G)$ are the corners at which $\gamma$ is marked in $\mathcal{F}$. We define the $\textbf{c}$\textbf{-vector} of $\gamma$ with respect to $\mathcal{F}$ to be $\textbf{c}_\mathcal{F}(\gamma) := \sum_{e\in s_{\gamma,\mathcal{F}}}\textbf{x}_e \in \mathbb{Z}^n$ (resp., $\textbf{c}_\mathcal{F}(\gamma) := -\sum_{e\in s_{\gamma,\mathcal{F}}}\textbf{x}_e \in \mathbb{Z}^n$) if $\gamma$ is green (resp., red). Note that the \textbf{c}-vector of $\gamma$ depends on the choice of facet $\mathcal{F}$ containing $\gamma$, whereas the \textbf{g}-vector $\textbf{g}(\gamma)$ is intrinsic to $\gamma$. We also let $C(\mathcal{F}) := \{\textbf{c}_\mathcal{F}(\gamma)\}_{\gamma\in \mathcal{F}^{red}\sqcup \mathcal{F}^{gr}}$. 

As the following proposition shows, the \textbf{c}-vectors $C(\mathcal{F})$ and the \textbf{g}-vectors $G(\mathcal{F})$ defined by a given facet are dual bases of $\mathbb{R}^n$.

\begin{proposition}\label{vanish}\cite[Proposition 22]{manneville2017geometric}\label{dualbases_prop}
For any $\gamma, \delta \in \mathcal{F}$ we have $\langle \textbf{g}(\delta), \textbf{c}_\mathcal{F}(\gamma) \rangle \in \{0,1\}$ and equals 1 if and only if $\gamma = \delta$.
\end{proposition}

\begin{example}\label{a2_vectors_ex}
Consider the tree in Figure~\ref{A2_vecs} where $\text{Int}(E_T)=\{e_1,e_2\}$. The \textbf{g}- and \textbf{c}-vectors associated to the facet in this figure are as follows:

\begin{center}
\begin{minipage}[c]{.4\textwidth}
\begin{tabular}{l l}
$\mathbf{g}(\gamma)=(-1,0)$ & $\mathbf{c_\mathcal{F}}(\gamma)= (-1, -1)$\\
$\mathbf{g}(\delta)=(-1,1)$ & $\mathbf{c_\mathcal{F}}(\delta)= (0, 1).$\\
\end{tabular}
\end{minipage}
\end{center}  
\end{example}

\begin{figure}[h]
$$\includegraphics[scale=1.4]{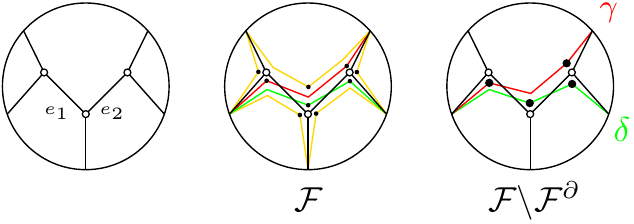}$$
\caption{\footnotesize{The tree and facet $\mathcal{F}$ of $\Delta^{NC}(T)$ from Example~\ref{a2_vectors_ex}.}}
\label{A2_vecs}
\end{figure}

We end this section with a lemma that we will interpret representation-theoretically in the next section. For $s_{\gamma, \mathcal{F}} = (v_0,\ldots, v_t)$, let $C_{s_{\gamma,\mathcal{F}}}$ denote the set of acyclic paths $s = (v_i,\ldots,v_j)$ such that by orienting $s_{\gamma, \mathcal{F}}$ from $v_0$ to $v_t$ one has that

$\bullet$ if $i>0$ then $s$ turns right at $v_i$, and

$\bullet$ if $j<t$ then $s$ turns left at $v_j$.

\noindent Observe that $C_{s_{\gamma, \mathcal{F}}}\backslash\{s_{\gamma,\mathcal{F}}\} $ is non-empty if and only if $s_{\gamma, \mathcal{F}}$ contains at least two edges. An example of the acyclic paths in $C_{s_{\gamma,\mathcal{F}}}$ is shown in Figure~\ref{subseg}.

\begin{lemma}\label{lem1}
Let $\mathcal{F}$ be a facet of $\Delta^{NC}(T)$ with at least one green arc, and let $\gamma \in \mathcal{F}^{red}$ be a red arc such that $s_{\gamma, \mathcal{F}}$ contains at least two edges of $T$. Then there exists a green arc $\delta \in \mathcal{F}^{gr}$ such that  $|Z_\delta^- \cap \{\text{edges of $t$}\}|=|Z_\delta^+ \cap \{\text{edges of $t$}\}|+1$ for any $t \in C_{s_{\gamma, \mathcal{F}}}\backslash \{s_{\gamma,\mathcal{F}}\}$. Moreover, for any arc $\delta \in \mathcal{F}^{gr}$ and any $t \in C_{s_{\gamma, \mathcal{F}}}$, we have that $|Z_\delta^- \cap \{\text{edges of $t$}\}| \geq |Z_\delta^+ \cap \{\text{edges of $t$}\}|$. 
\end{lemma}

\begin{figure}[h]
\includegraphics[width=.55\textwidth]{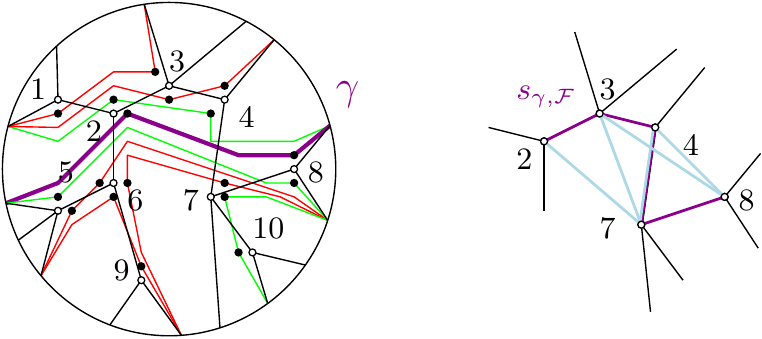}
\centering
\caption{\footnotesize{For $\gamma$ we have $C_{s_{{\gamma},\mathcal{F}}}=\{  [2,8], [2,7] , [3,8], [3,7], [4,8], [4,7]\}$ which appear in light blue, except for $s_{\gamma, \mathcal{F}} = [2,8]$ which is purple. Note that orienting $s_{\gamma, \mathcal{F}}$ from 2 to 8, we have that $s_{\gamma, \mathcal{F}}$ turns right at 3 and 4, and left at 7}.}
\label{subseg}
\end{figure}

\begin{proof}The statement follows from establishing three claims. In each of the following arguments, let $(v,F)$ and $(u,G)$ denote the corners at which $\gamma$ is marked, and orient $s_{\gamma, \mathcal{F}} = [v,u]$ from $v$ to $u$. At times, we will also write $s_{\gamma, \mathcal{F}} = (v_0, \ldots, v_t)$ where $v = v_0$ and $v_t = u$.

\textit{\textbf{Claim 1:}} Let $(v_i,H)$ be a corner contained in $\gamma$ for some $i \in \{1, \ldots, t-1\}$. If $H \in \text{Reg}(\gamma,G)$, then there exists a green arc $\delta$ containing the acyclic path $s_{\gamma, \mathcal{F}}$ and the corner $(v,F)$. Dually, if $H \in \text{Reg}(\gamma,F)$, then there exists a green arc $\delta$ containing the acyclic path $s_{\gamma, \mathcal{F}}$ and the corner $(u,G)$.



Let $\mu \lessdot_{(v,F)} \gamma$ and $\nu \lessdot_{(u,G)} \gamma$ be the supporting arcs of $\gamma$ and orient them so that they agree with the orientation of $s_{\gamma, \mathcal{F}}$. As $\mu$ and $\nu$ agree at $s_{\gamma, \mathcal{F}}$, they both contain $(v_i, H)$. Without loss of generality, suppose that $H \in \text{Reg}(\gamma, G)$. It follows that $\mu$ is not a boundary arc.

{Let $I$ be the face immediately counterclockwise from $G$ about $u$. The arc $\mu$ is marked at $(u, I)$ and it turns right at $u$, as $\gamma$ is a red arc. Let $(u', I')$ be the other corner at which $\mu$ is marked. If $u'$ comes before $u$ in the given orientation of $\mu$, then  $\mu$  is a green arc and the result holds.}

{Suppose now that $u'$ comes after $u$. Then $\mu$ is red. Set $\mu'$ to be the supporting arc of $\mu$ such that $\mu' \lessdot_{(u, I)} \mu$. Since $\mu'$ must be marked at $(u', I'')$ where $I''$ is immediately counterclockwise from $I'$ about $u'$, then $\mu$ and $\mu'$ separate at $u'$. By Lemma~\ref{supp_arcs_lemma}, $\mu'$ also contains $s_{\gamma, \mathcal{F}}$ and the corner $(v,F)$, since this is true of arc $\mu$. In particular, $\mu'$ contains $(v_i,H)$ and $H\neq I''$. Thus $\mu'$ is not a boundary arc.}

{Let $(w,J)$ denote the other corner at which $\mu'$ is marked and orient $\mu'$ so that it agrees with the orientation of $s_{\gamma, \mathcal{F}}$. Note that $\mu'$ must turn right at $u'$. As before, if $w$ comes before $u^\prime$, then $\mu'$ is green and the result holds. Otherwise, we repeat the above argument and find a non-boundary supporting arc of $\mu'$ containing $s_{\gamma, \mathcal{F}}$ and the corner $(v,F)$. Since there are finitely many arcs in $\mathcal{F}$, this process must stop at some green arc $\delta$ containing $s_{\gamma, \mathcal{F}}$ and the corner $(v,F)$.}

\textit{\textbf{Claim 2:}} There exists an arc $\delta \in \mathcal{F}^{gr}$ that satisfies $\{ \text{edges in $t$}\} \cap Z_{\delta} \neq \emptyset$  and   $|Z_\delta^- \cap \{\text{edges of $t$}\}|=|Z_\delta^+ \cap \{\text{edges of $t$}\}|+1$ for any $t \in C_{s_{\gamma, \mathcal{F}}}\backslash \{s_{\gamma,\mathcal{F}}\}$. 

{Let $t= [v_i, v_j]\in C_{s_{\gamma, \mathcal{F}}}\backslash \{s_{\gamma,\mathcal{F}}\}$ where vertex $v_i$ comes before $v_j$ according to the orientation of $s_{\gamma,\mathcal{F}}$. We prove the assertion in each of the following three cases.}

\textbullet \quad $v_i \neq v$ and $v_j \neq u$

{Since we know that $t \in C_{s_{\gamma, \mathcal{F}}}$, the arc $\gamma$ turns left at $v_j$ and right at $v_i$. In particular, $\gamma$ contains a corner $(v_j, H)$ where $H \in \text{Reg}(\gamma, G)$. By Claim 1, this implies that there exists a green arc $\delta$ containing $s_{\gamma,\mathcal{F}}$ and containing the corner $(v,F).$ Since $\delta$ contains $s_{\gamma,\mathcal{F}}$, we have that $\delta$ also turns left at $v_j$ and right at $v_i$. That is,  $Z_{\delta} \cap \{\text{edges in $t$}\}  \neq \emptyset$.}

{Now let $e$ and $e'$  be the first and last edges of $t$ contained in $Z_{\delta} \cap \{ \text{edges in $t$}\}$ with respect to the chosen orientation of $s_{\gamma,\mathcal{F}}$. Since the arc $\delta$ turns right at $v_i$, we have that $g_{\delta}^{e}=-1$. Similarly, since $\delta$ turns left at $v_j$, the last change of direction must be from right to left and again $g_{\delta}^{e'}=-1$. Observe that the coordinates of $\textbf{g}(\delta)$ associated to edges in $Z_{\delta} \cap \{\text{edges of $t$}\} $ alternate in sign when ordered in a way that is consistent with the orientation of $s_{\gamma, \mathcal{F}}$. Moreover, the first and last of these nonzero coordinates of $\textbf{g}(\delta)$ are $-1$. Thus $|Z_\delta^- \cap \{\text{edges of $t$}\}|=|Z_\delta^+ \cap \{\text{edges of $t$}\}|+1$.}

\textbullet \quad $v_i=v$ and $v_j \neq u$

{As in the previous case, the arc $\gamma$ contains a corner $(v_j, H)$ where $H \in \text{Reg}(\gamma, G)$. Therefore, Claim 1 implies that there exists a green arc $\delta$ containing $s_{\gamma,\mathcal{F}}$ and containing the corner $(v,F).$ Since $\delta$ contains $s_{\gamma,\mathcal{F}}$ and the corner $(v,F)$, we have that $\delta$ also turns left at $v_j$ and right at $v_i$. The remainder of the argument in the previous case may now be applied to this case.}

\textbullet \quad $v_i \neq v$ and $v_j = u$

{By the definition of $C_{s_{\gamma, \mathcal{F}}}$ and that $t \in C_{s_{\gamma, \mathcal{F}}}$, the arc $\gamma$ turns right at $v_i$. In particular, $\gamma$ contains a corner $(v_i, H)$ where $H \in \text{Reg}(\gamma, F)$. By the dual statement in Claim 1, this implies that there exists a green arc $\delta$ containing $s_{\gamma,\mathcal{F}}$ and containing the corner $(u,G).$ Since $\delta$ contains $s_{\gamma,\mathcal{F}}$ and the corner $(u,G)$, we have that $\delta$ also turns left at $v_j$ and right at $v_i$. One may now adapt the argument of the first case to this case.}

\textit{\textbf{Claim 3:}} For any  arc $\delta \in \mathcal{F}$ and any $t \in C_{s_{\gamma,\mathcal{F}}}$, we have that $|Z_\delta^- \cap \{\text{edges of $t$}\}| \geq |Z_\delta^+ \cap \{\text{edges of $t$}\}|$.

If $Z_\delta\cap \{\text{edges of $t$}\} = \emptyset$, then the result holds. Therefore, we assume this intersection is non-empty.

Let $t = (v_i,\ldots, v_j)$, and let $(u^\prime_1,u_1)$ (resp., $(u_2,u^\prime_2)$) denote the first (resp., last) edges in $Z_\delta\cap \{\text{edges of $t$}\}$ with respect to the chosen orientation of $s_{\gamma, \mathcal{F}}$. Since the coordinates of $\textbf{g}(\gamma)$ associated to edges in $Z_{\delta} \cap \{\text{edges of $t$}\}$ alternate in sign when ordered in a way that is consistent with the orientation of $s_{\gamma, \mathcal{F}}$, it is enough to show that $\textbf{g}(\delta)$ does not satisfy $g_\delta^{(u^\prime_1,u_1)} = g_\delta^{(u_2,u^\prime_2)}=1$. Since $t \in C_{s_{\gamma,\mathcal{F}}}$, we see that the desired result holds when $\delta = \gamma$.

Now suppose that $g_\delta^{(u^\prime_1,u_1)} = g_\delta^{(u_2,u^\prime_2)}=1$ for some arc $\delta \in \mathcal{F}\backslash\{\gamma\}$. Orient $\delta$ in a way that is consistent with the orientation of $s_{\gamma, \mathcal{F}}$. Let $(v_i,F_i)$ and $(v_j,F_j)$ be corners of $T$ contained in $\gamma$. Since $t \in C_{s_{\gamma,\mathcal{F}}}$, there exists a face $G_i$ (resp., $G_j$) that is immediately counterclockwise from $F_i$ about $v_i$ (resp., from $F_j$ about $v_j$), and where $F_i$ and $G_i$ (resp., $F_j$ and $G_j$) are both incident to the edge $(v_i,v_{i+1})$ (resp., $(v_{j-1},v_j)$).

Since $(u^\prime_1,u_1)$ is the first element of $Z_\delta\cap \{\text{edges of $t$}\}$ contained in $\delta$ and since $\delta$ turns left at $u_1^\prime$, we know that $\delta$ turns left at each vertex in $\{v_i, v_{i+1},\ldots, u^\prime_1\}$. Similarly, $\delta$ turns right at each vertex in $\{u^\prime_2, \ldots, v_{j-1}, v_j\}$. Consequently, we obtain that $\delta$ contains the corners $(v_i,G_i)$ and $(v_j,G_j)$ (see Figure~\ref{claim3_figure}). We obtain that $\text{Reg}(\gamma,F_i) \not \subset \text{Reg}(\delta,F_i)$, which implies that $\delta$ and $\gamma$ are crossing, a contradiction.
\end{proof}

\begin{figure}
$$\includegraphics[scale=1.3]{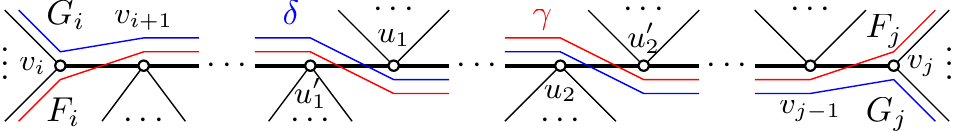}$$
\caption{The configuration of the arcs $\gamma$ and $\delta$ from Claim 3.}
\label{claim3_figure}
\end{figure}

\section{Tiling algebras}\label{sec_til_alg}

We now recall how a tree gives rise to a finite dimensional algebra. Given a tree $T$, let $Q_T$ be the quiver whose vertex set is $\text{Int}(E_T)$ and where $e, e' \in \text{Int}(E_T)$ are connected by an arrow in $Q_T$ if they meet in a corner of $T$. By convention, 
$e \: \overset{\alpha}{\rightarrow} \: e'$
if and only if $e'$ is immediately counterclockwise from $e$ about their common vertex. {That is, there is an injective map from the set of arrows of $Q_T$ to the set of corners of $T$.} {We define $I_T \subset \Bbbk Q_T$ to be the ideal generated by the relations $\alpha \beta$ where $\alpha : e_2 \rightarrow e_3$ (resp., $\beta : e_1 \rightarrow e_2$) corresponds to the corner $(v,F)$ (resp., $(v,G)$) and the face $F$ is immediately counterclockwise from the face $G$ about $v$.}

We define the \textbf{tiling algebra} of $T$ to be $\Lambda_T := \Bbbk Q_T/I_T$ where $\Bbbk$ is an algebraically closed field. Tiling algebras are a family of representation finite gentle algebras that were introduced in \cite{simoes2017endomorphism}. We invite the reader to check that $\text{dim}_\Bbbk\Lambda_T = 3$ (resp., 26) when $T$ is the tree from Figure~\ref{A2_vecs} (resp., Figure~\ref{tree_quiver}).

The category of finitely generated left modules over $\Lambda_T$ is equivalent to category of finite dimensional representations of $Q_T$ over $\Bbbk$ that are \textbf{compatible} with the relations from  $I_T$ (i.e., a representation $M=((M_i)_{i \in {Q_T}_0}, (\varphi_\alpha)_{\alpha \in {Q_T}_1})$ of $Q_T$ where for any $\sum_{i = 1}^k c_i\alpha^{(i)}_{1}\cdots \alpha^{(i)}_{\ell_i} \in I_T$ where $c_i \in \Bbbk$ for all $i \in \{1,\ldots, k\}$ we have that $\sum_{i = 1}^k c_i\varphi_{\alpha^{(i)}_{1}}\cdots \varphi_{\alpha^{(i)}_{\ell_i}} = 0$).\footnote{For a general finite dimensional $\Bbbk$-algebra $\Lambda = \Bbbk Q/I$ where $I$ is an admissible ideal, one can also equivalently describe modules over $\Lambda$ as representations of $Q$ compatible with $I$.} In the case of tiling algebras, the compatibility condition becomes $\varphi_{\alpha}\varphi_{\beta} = 0$ for all $\alpha\beta\in I_T$. 

We also know that the indecomposable $\Lambda_T$-modules are \textbf{string modules}, which we denote by $M(w)$. This follows from the fact that tiling algebras are gentle algebras, which was first observed in \cite[Proposition 3.2]{simoes2017endomorphism}. Let $({Q_T}_1)^{-1}$ denote the set of formal inverses of arrows of $Q_T$. Given $\alpha \in {Q_T}_1$, an arrow of $Q_T$, let $\textsf{s}(\alpha)$ (resp., $\textsf{t}(\alpha)$) denote the source and target of the arrow $\alpha$. Similarly, for any $\alpha^{-1} \in ({Q_T}_1)^{-1}$ define
$\textsf{s}(\alpha^{-1}):=\textsf{t}(\alpha)$ and $\textsf{t}(\alpha^{-1}):=\textsf{s}(\alpha)$.  A \textbf{string} in $\Lambda_T$ is a word $w = \alpha_d^{\epsilon_d}\cdots\alpha_1^{\epsilon_1}$ in the alphabet ${Q_T}_1 \sqcup ({Q_T}_1)^{-1}$ with $\epsilon_i \in \{\pm 1\}$, for all $i \in \{1,2,\ldots,d\}$, which satisfies the following conditions:

\begin{enumerate}
\item $\textsf{s}(\alpha_{i+1}^{\epsilon_{i+1}})= \textsf{t}(\alpha_i^{\epsilon_i})$ and $ \alpha_{i+1}^{\epsilon_{i+1}} \neq \alpha_i^{-\epsilon_i}$, for all $i\in \{1,\ldots,d-1 \}$, and
\item $w$ and also $w^{-1}:= \alpha_1^{-\epsilon_1}\cdots\alpha_d^{-\epsilon_d}$ do not contain a subpath in $I$.
\end{enumerate}
For each vertex $i \in {Q_T}_0$, there is also a string whose string module is the unique simple representation of $Q_T$ supported only at vertex $i$. By abuse of notation, we write $w = i$ where $i \in {Q_T}_0$ for such strings. In other words, a string is an irredundant walk in $Q_T$ that avoids the relations in $I_T$.

Let $w = \alpha_d^{\epsilon_d}\cdots\alpha_1^{\epsilon_1}$ be a string in $\Lambda_T$. In analogy with the above definition, define $\textsf{s}(w):= \textsf{s}(\alpha_1^{\epsilon_1})$ and $\textsf{t}(w):= \textsf{t}(\alpha_d^{\epsilon_d})$. In the case where $w = i$ for some $i \in {Q_{T}}_0$, we set $\textsf{s}(w) := i$ and $\textsf{t}(w) := i.$ It will also be useful to define a \textbf{substring} $w^\prime$ of $w$ as a string of the form $w^\prime = \alpha_k^{\epsilon_k}\cdots \alpha_j^{\epsilon_j}$ with $1 \le j \le k \le d$ or $w^\prime = i$ where $i$ is a vertex appearing in $w$.

In the setting of tiling algebras, all strings are supported on connected acyclic subgraphs of $Q_T$, but this is not the case in general. The \textbf{string module} $M(w)$ is the representation of $Q_T$ obtained by assigning the vector space $\Bbbk$ to each vertex in the string $w$ and identity morphisms to each arrow in $w$.\footnote{For simplicity, we have given the definition of $M(w)$ only in the generality of tiling algebras.} 

Using these facts, we obtain that the indecomposable $\Lambda_T$-modules are parameterized by \textbf{segments} of $T$ (i.e., acyclic paths $s=(v_0, v_1,...,v_t)$ whose endpoints are interior vertices of $T$ and any two consecutive edges $(v_{i-1},v_i)$ and $(v_i,v_{i+1})$  are incident to a common face) \cite[Corollary 4.3]{garver_oriented_rep_thy}. Let $w(s)$ denote the unique string in $\Lambda_T$ corresponding to the segment $s \in \text{Seg}(T)$, the set of all segments of $T$, and $M(w(s))$ the corresponding string module.  We show an example of this bijection in Figure~\ref{ind_to_segs}. Using this bijection, one obtains the following lemma.

\begin{lemma}\label{lem2}
Given a facet $\mathcal{F}$ of $\Delta^{NC}(T)$ and $\gamma \in \mathcal{F}$, the set map $C_{s_{\gamma, \mathcal{F}}} \to \text{mod}(\Lambda_T)$ defined by $t \mapsto M(w(t))$ induces a bijection between elements of $C_{s_{\gamma, \mathcal{F}}}$ and the indecomposable submodules of $M(w(s_{\gamma, \mathcal{F}})).$
\end{lemma}

\begin{proof}
The data of a proper indecomposable submodule $M(w)$ of $M(w(s_{\gamma,\mathcal{F}}))$ is equivalent to the data of a substring $w$ of $w(s_{\gamma,\mathcal{F}})$ with the property that there exist strings $w^1 = \alpha_d^{\epsilon_d}\cdots \alpha_1^{\epsilon_1}$ and $w^2 = \beta_c^{\epsilon_c}\cdots \beta_1^{\epsilon_1}$ in $\Lambda_T$ (at most one of which may be the empty string) such that $w(s_{\gamma, \mathcal{F}}) = w^1ww^2$ where $\textsf{s}(\alpha_1^{\epsilon_1}) = \textsf{t}(w)$ and $\textsf{t}(\beta_c^{\epsilon_c}) = \textsf{s}(w)$. Note that the bijection between segments of $T$ and indecomposable $\Lambda_T$-modules implies that there is a unique segment $t \in \text{Seg}(T)$ such that $w = w(t).$ The condition that $\textsf{s}(\alpha_1^{\epsilon_1}) = \textsf{t}(w)$ and $\textsf{t}(\beta_c^{\epsilon_c}) = \textsf{s}(w)$ is equivalent to the arc $\gamma$ turning at the endpoints of $t$ in such a way that $t \in C_{s_{\gamma, \mathcal{F}}}.$\end{proof}

\begin{figure}
$$\begin{array}{cccc} \includegraphics[scale=1.2]{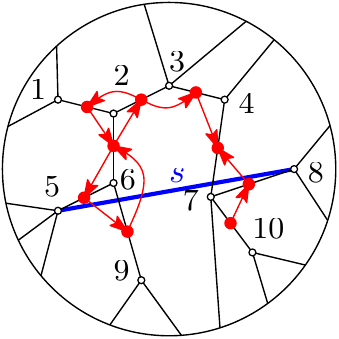} & \includegraphics[scale=1.2]{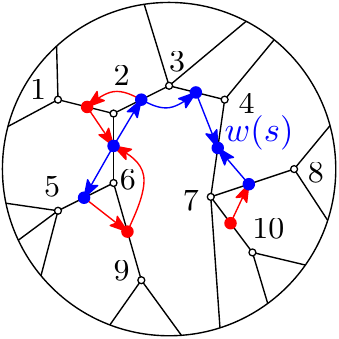} & \raisebox{1.2in}{\begin{xy} 0;<1pt,0pt>:<0pt,-1pt>:: 
(0,0) *+{0} ="0",
(30,0) *+{\Bbbk} ="1",
(15,30) *+{\Bbbk} ="2",
(0,60) *+{\Bbbk} ="3",
(30,75) *+{0} ="4",
(60,0) *+{\Bbbk} ="5",
(75,30) *+{\Bbbk} ="6",
(90,60) *+{\Bbbk} ="7",
(65,75) *+{0} ="8",
"1", {\ar_0"0"},
"0", {\ar_0"2"},
"2", {\ar_1"1"},
"1", {\ar^1"5"},
"2", {\ar_1"3"},
"4", {\ar_0"2"},
"3", {\ar_0"4"},
"5", {\ar^1"6"},
"7", {\ar_1"6"},
"8", {\ar_0"7"},
\end{xy}}  \\ {segment \ s = [5,8]} & {string \ w(s)} &  {string \ module \ M(w(s))} \end{array}$$
\caption{A segment $s = [5,8]$, the string $w(s)$ that it defines, and the corresponding string module $M(w(s))$.}
\label{ind_to_segs}
\end{figure}

\begin{remark}\label{remark_Ks}
Let $s = (v_0,\ldots, v_t)$ be any segment of $T$. One defines $K_s \subset \text{Seg}(T)$ to be the collection of segments $(v_i,\ldots, v_j)$ such that by orienting $s$ from $v_0$ to $v_t$ one has that
\begin{itemize}
\item if $i>0$ then $s$ turns left at $v_i$, and
\item if $j < t$ then $s$ turns right at $v_j$.
\end{itemize}
A small modification to the proof of Lemma~\ref{lem2} shows that the set map $K_s \to \text{mod}(\Lambda_T)$ defined by $t \mapsto M(w(t))$ induces a bijection between elements of $K_s$ and the indecomposable quotient modules of $M(w(s))$. 
\end{remark}

\section{Noncrossing tree partitions}\label{Sec_noncrossing_tree_part}
Now let $V^{\circ}_T$ denote the set of interior vertices of $T$, and choose a real number $\epsilon > 0$ so that the balls of radius $\epsilon$ around each vertex in $V^\circ_T$ do not intersect each other and are contained in $D^2$. Furthermore, we require that any $\epsilon$-ball centered at a vertex $v$ may only contain points from edges of $T$ that have $v$ as an endpoint. For each corner $(v,F)$, fix a point $z(v,F)$ in the interior of $F$ such that $d(z(v,F),v)=\epsilon$ where $d(- , -)$ is the usual Euclidean metric on $D^2$. Let $T_\epsilon \subset D^2$ denote the set of points belonging to the embedded tree $T$ plus the set of points belonging to the $\epsilon$-ball around some interior vertex. More explicitly, \[T_{\epsilon}=T\cup\bigcup_{v\in V^o}\{x\in D^2:\ d(x,v) < \epsilon\}.\] Additionally, for any $s \in \text{Seg}(T)$, let $s_\epsilon$ denote the set of points on an edge of $s$ whose distance is at least $\epsilon$ from any interior vertex of $T$.

Let  $(v,e,F)$ and $(u,e',G)$  be two green (resp., red) flags such that $[v,u]$ is a segment, a \textbf{green} (resp., \textbf{red}) \textbf{admissible curve} for $[v,u]$ is a simple curve $\sigma:[0,1] \rightarrow D^2$ for which $\sigma(0)=z(v,F)$, $\sigma(1)=z(u,G),$ and $\sigma([0,1])\subseteq D^2\backslash(T_\epsilon\backslash[u,v]_\epsilon)$.\footnote{Up to coloring-preserving isotopy relative to $z(v,F)$ and $z(u,G)$, there is a unique green (resp., red) admissible curve for $[v,u]$.} Strictly speaking, the endpoints of $\sigma$ are $z(v,F)$ and $z(u,G)$. However, for convenience, we will often refer to the endpoints of $\sigma$ as simply $v$ and $u$ when it is not necessary to specify the exact corners where $\sigma$ begins and ends.

Two segments are \textbf{noncrossing} if they admit admissible curves that do not intersect each other, otherwise they are \textbf{crossing}. A segment is   \textbf{green} (resp., \textbf{red}) if it is represented by a green (resp., red) admissible curve. For $B\subset V^\circ_T$, let Seg$_r(B)$ be the set of \textbf{inclusion-minimal} red segments whose endpoints lie in $B$. That is, there do not exist distinct segments $s, t \in \text{Seg}_r(B)$ where $t$ is a subsegment of $s$. We define $\text{Seg}_g(B)$ to be the set of \textbf{inclusion-minimal} green segments whose endpoints lie in $B$.

A \textbf{noncrossing tree partition} $\textbf{B}=\{{B}_1,{B}_2,...,{B}_\ell\}$ is a set partition of $V^\circ_T$ such that any two segments of Seg$_r(\textrm{\textbf{B}}):=\bigcup_{i=1}^{l} \textnormal{Seg}_r(B_i)$ are noncrossing and each block of \textbf{B} is \textbf{segment-connected} (i.e., for any two vertices in $B_i$ there exists a sequence of segments in $\text{Seg}_r(B_i)$ that joins them). Let NCP($T$) be the poset of noncrossing tree partitions of $T$ ordered by refinement. Alternatively, we could have defined noncrossing tree partitions using the set $\text{Seg}_g(\textbf{B}) := \bigcup_{i=1}^{l} \textnormal{Seg}_g(B_i)$, and this definition produces the same lattice of noncrossing tree partitions.

Returning to tiling algebras, a full, additive subcategory $\mathcal{W} \subset \textnormal{mod}(\Lambda_T)$ is a \textbf{wide subcategory} if it is exact abelian and \textbf{extension closed} (i.e., for any short exact sequence $0\to X \to Z \to Y \to 0$ with $X,Y \in \mathcal{W}$, one has that $Z \in \mathcal{W}$). We remark that all subcategories we work with in this paper are assumed to be closed under isomorphisms. Under this assumption, the collection of all wide subcategories of $\text{mod}(\Lambda_T)$, denoted $\text{wide}(\Lambda_T)$, is a set. In fact, $\text{wide}(\Lambda_T)$ is a poset under inclusion. The intersection of two wide subcategories is a wide subcategory, and the zero subcategory (resp., mod($\Lambda_T$)) is the unique minimal (resp., unique maximal) element of wide$(\Lambda_T)$. As  $\Lambda_T$ is representation finite, the poset wide$(\Lambda_T)$ is finite. Therefore, wide$(\Lambda_T)$ is a lattice.

In \cite[Theorem 7.1]{garver_oriented_rep_thy}, the second author and McConville obtained a poset isomorphism between the lattice of noncrossing tree partitions and the lattice of wide subcategories given by \[\begin{array}{rcl}\rho : \textnormal{NCP}(T) &\longrightarrow & \textnormal{wide}(\Lambda_T) \\
\textnormal{\textbf{B}} & \longmapsto &\text{add}\left(\displaystyle \bigoplus_{w(s)} M(w(s)) \ | \ s \in \overline{\text{Seg}_r(\textnormal{\textbf{B}})} \right)\end{array}\]
where $\overline{\text{Seg}_r(\textnormal{\textbf{B}})} \subset \text{Seg}(T)$ is the \textbf{closure} of $\text{Seg}_r(\textbf{B})$ (i.e., the smallest set of segments containing $\text{Seg}_r(\textrm{\textbf{B}})$ such that $s = (v_0, v_1, \ldots, v_k),$ $t = (v_k, v_{k+1} \ldots, v_\ell) \in \text{Seg}_r(\textrm{\textbf{B}})$ and ${s\circ t} \in \text{Seg}(T)$ implies $s\circ t \in \overline{\text{Seg}_r(\textnormal{\textbf{B}})}$). The category $\text{add}(M)$ where $M \in \text{mod}(\Lambda_T)$ is defined as the smallest full, additive subcategory of $\text{mod}(\Lambda_T)$ that contains $M$ and is closed under taking direct summands. We show in Figure~\ref{bij_wide_nctp} an example of this isomorphism.

\begin{figure}
$$\includegraphics[scale=0.7]{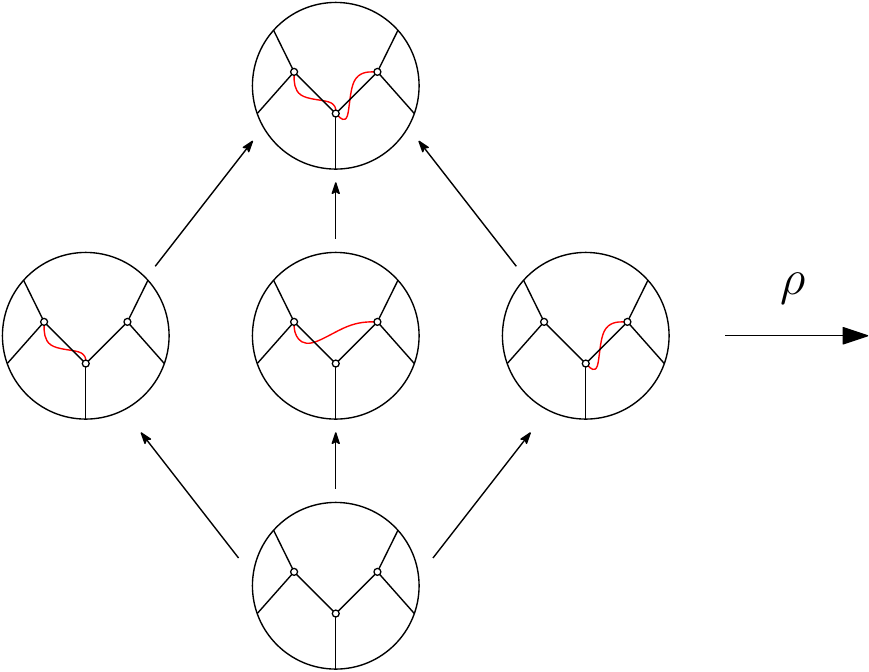}
\hspace{2mm} \includegraphics{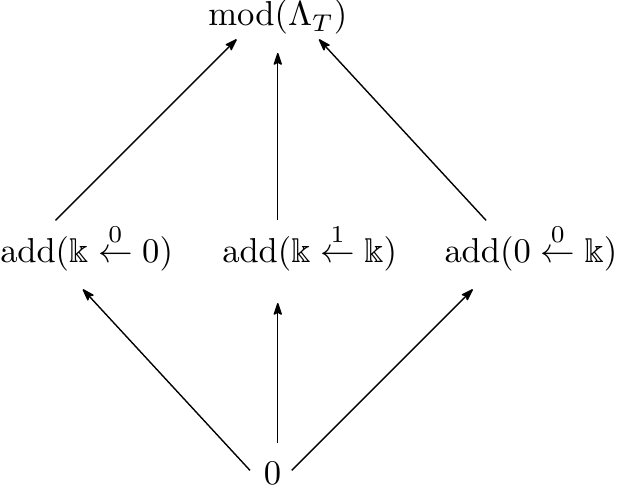}$$
\caption{The lattice of wide subcategories of mod$(\Lambda_T)$ and its corresponding lattice of noncrossing tree partitions, which we present using sets of red admissible curves.}
\label{bij_wide_nctp}
\end{figure}

Now fix a facet $\mathcal{F} \in \Delta^{NC}(T)$. For each $\delta \in \mathcal{F}\backslash\mathcal{F}^\partial$, let $\sigma$ be an admissible curve for $s_{\delta,\mathcal{F}}$ whose color is the same as the color of $\delta$. {We define $\psi_r(\mathcal{F})$ (resp., $\psi_g(\mathcal{F})$) to be the set partition of $V^\circ_T$ whose blocks consist of interior vertices connected by a sequence of these red (resp., green) admissible curves (see Figure~\ref{nc_to_ncp}).} The map $\psi_r$ is a bijection between the facets of $\Delta^{NC}(T)$ and $\text{NCP}(T)$. In fact, it induces a bijection $\text{Kr}: \text{NCP}(T) \to \text{NCP}(T)$ defined by $\text{Kr}(\psi_r(\mathcal{F})):= \psi_g(\mathcal{F})$. We refer to $\text{Kr}(\textbf{B})$ with $\textbf{B} \in \text{NCP}(T)$ as the \textbf{Kreweras complement} of \textbf{B}. We remark that using the map $\rho$, we can regard the Kreweras complementation map as a cyclic action on wide subcategories of $\text{mod}(\Lambda_T)$.

\begin{figure}[!htbp]
$$\includegraphics[scale=1.2]{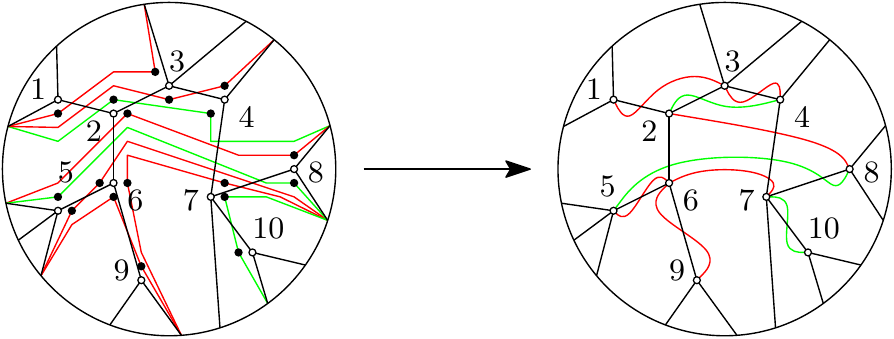}$$
\caption{On the left, we show the red and green arcs of a facet $\mathcal{F}$ of $\Delta^{NC}(T)$. On the right, we show the red-green tree $\mathcal{T}_\mathcal{F}.$ Here $\psi_r(\mathcal{F}) = \{\{1,3,4\}, \{2,8\}, \{5,6,7,9\}, \{10\}\}$ and $\text{Kr}(\psi_r(\mathcal{F})) = \psi_g(\mathcal{F}) = \{\{1\}, \{2,4\}, \{3\}, \{5,8\}, \{6\}, \{7,10\}, \{9\}\}$.}
\label{nc_to_ncp}
\end{figure}

Given $\textbf{B} \in \text{NCP}(T)$ corresponding to facet $\mathcal{F}$, choose a set $S_r$ (resp., $S_g$) of pairwise noncrossing red (resp., green) admissible curves realizing the segments in $\text{Seg}_r(\textbf{B})$ (resp., $\text{Seg}_g(\text{Kr}(\textbf{B}))$). By  {\cite[Theorem 5.11]{garver2016oriented}}, the sets $S_r$ and $S_g$ define a new tree $\mathcal{T}_\mathcal{F}$ whose vertex set is $V^\circ_T$ and whose edge set is $S_r \sqcup S_g$. We refer to $\mathcal{T}_\mathcal{F}$ as the \textbf{red-green tree} of $\mathcal{F}$.

Given a red-green tree $\mathcal{T}_\mathcal{F}$ and an acyclic path $s = (v_0,\ldots, v_t)$ of $T$ where $v_0, v_t \in V^\circ_T$, it will be useful to define $\varsigma(s) = (\sigma_{i_1},\ldots, \sigma_{i_\ell})$ to be the unique shortest sequence of admissible curves in $\mathcal{T}_\mathcal{F}$ joining $v_0$ and $v_t$ where $\sigma_{i_j}$ and $\sigma_{i_{j'}}$ with $j\neq j'$ share an endpoint if and only if $j' = j \pm 1$. By convention, we assume that $v_0$ is an endpoint of $\sigma_{i_1}$ and $v_t$ is an endpoint of $\sigma_{i_\ell}$. We will also need the following lemma for a result in the next section.

\begin{lemma}\label{lemma_forbidden_curves}
Let $\mathcal{T}_\mathcal{F}$ be a red-green tree, and let $s = (v_0,\ldots, v_t)$ be an acyclic path of $T$ where $v_0, v_t \in V^\circ_T$. Up to isotopy, there is no admissible curve $\xi$ nor $\xi'$ as in Figure~\ref{no_sub_no_quotient} that appears in the sequence $\varsigma(s)$.
\end{lemma}

\begin{proof}
Let $(s_{i_1}, \ldots, s_{i_\ell})$ denote the sequence of segments of $T$ where for each $j \in \{1,\ldots, \ell\}$ the segment $s_{i_j}$ is represented by the admissible curve $\sigma_{i_j}$.

Suppose there exists a curve in the sequence $\varsigma(s) = (\sigma_{i_1},\ldots, \sigma_{i_\ell})$ of the form $\xi$ or $\xi'$ as in Figure~\ref{no_sub_no_quotient}. Without loss of generality, we assume the former. The curve $\xi$ must be equal to $\sigma_{i_1}$. We show that $\sigma_{i_1}$ cannot equal $\xi$ as in Figure~\ref{no_sub_no_quotient} $(iii)$ under the assumption that $\sigma_{i_1}$ is green. The proofs in the other cases are analogous.

\begin{figure}[!htbp]
$$\begin{array}{cccccccc}\includegraphics[scale=1.75]{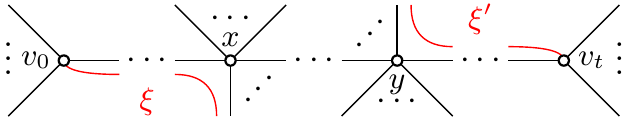} \\ (i) \\ \includegraphics[scale=1.75]{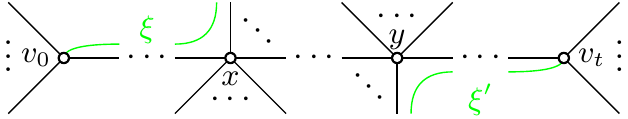}\\ (ii) \\
\includegraphics[scale=1.75]{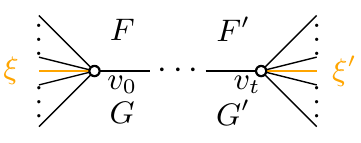} \\ (iii)
\end{array}$$
\caption{Several types of curves that cannot appear in the sequence $\varsigma(s)$. In $(i)$ and $(ii)$, $x$ (resp., $y$) is a vertex of the acyclic path $[v_0,v_t]$ other than $v_0$ (resp., $v_t$). In $(i)$ and $(ii)$, the curves $\xi$ and $\xi^\prime$ are not allowed to have both endpoints be vertices of $[v_0,v_t]$. In $(iii)$, the curves $\xi$ and $\xi^\prime$ are drawn in brown because they could be either red or green. Also in $(iii)$, the curve $\xi$ (resp., $\xi^\prime$) is not allowed to have an endpoint of the form $z(v_0,F)$ or $z(v_0,G)$ (resp., $z(v_t,F^\prime)$ or $z(v_t,F^\prime)$).}
\label{no_sub_no_quotient}
\end{figure}


We first assume that the colors of the curves in the sequence $\varsigma(s)$ are alternating. If $s_{i_1}$ and $s_{i_2}$ have a common subsegment, then the curve $\sigma_{i_2}$ appears in the configuration shown in Figure~\ref{first_case_fig} $(a)$ up to isotopy. If $s_{i_1}$ and $s_{i_2}$ have no common subsegment, then the curve $\sigma_{i_2}$ appears in the configuration shown in Figure~\ref{first_case_fig} $(b)$ up to isotopy.


\begin{figure}[!htbp]
$$\begin{array}{ccccccccccc}
\includegraphics[scale=1.65]{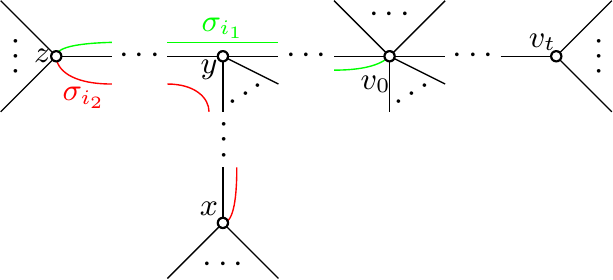} \\ (a) \\ \includegraphics[scale=1.65]{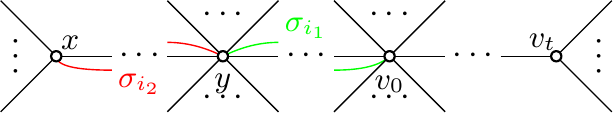}\\ (b)\\
\end{array}$$
\caption{In $(a)$, we show the case when $s_{i_1} = [v_0,z]$ and $s_{i_2} = [z,x]$ have a common subsegment. The segment $[z,y]$ is the unique longest segment contained in both $s_{i_1}$ and $s_{i_2}$. In $(a)$, we allow for the case where either $y=x$ or $y = v_0$. The case where $x = v_0$ cannot happen because $s_{i_1} \neq s_{i_2}$. In $(b)$, we show the case when $s_{i_1}= [v_0,y]$ and $s_{i_2}= [y,x]$ have no common subsegment. In $(b)$, vertex $y$ is the unique common vertex of $s_{i_1}$ and $s_{i_2}$.}
\label{first_case_fig}
\end{figure}


Suppose that $s_{i_1}$ and $s_{i_2}$ have a common subsegment $[z, y]$ in the notation of Figure~\ref{first_case_fig} $(a)$. If vertex $y$ or vertex $v_0$ has degree greater than 3, then the acyclic path $[x,v_t]$ is not a segment of $T$. This implies that there does not exist an admissible curve connecting $x$ and $v_t$. On the other hand, assume that vertex $y$ and vertex $v_0$ both have degree equal to 3. In this case, even though $[x,v_t]$ is now a segment, any admissible curve connecting $x$ and $v_t$ must cross $\sigma_{i_2}$ (see Figure~\ref{first_case_fig} $(a)$). 

We reduce to considering the behavior of the green admissible curve $\sigma_{i_3}$. A similar analysis shows that the common endpoint of $s_{i_3}$ and $s_{i_4}$ cannot be connected to $v_t$ by an admissible curve that is pairwise noncrossing with the curves appearing in $\varsigma(s)$. By continuing this process, we see that the sequence of curves $\varsigma(s)$ cannot return to vertex $v_t$ in such a way that the curves are pairwise noncrossing. An analogous argument reaches the same contradiction when $s_{i_1}$ and $s_{i_2}$ have no common subsegment. 

We now explain the general case where the colors of the curves in the sequence $\varsigma(s)$ are not necessarily alternating. Let $\sigma_{i_1}, \ldots, \sigma_{i_j}$ (resp., $\sigma_{i_{j+1}}, \ldots, \sigma_{i_{k}}$) denote the first $j$ green (resp., first $k$ red) admissible curves in the sequence $\varsigma(s)$. By the defining properties of noncrossing tree partitions, no two segments in the sequence $s_{i_1}, \ldots, s_{i_j}$ have a common subsegment, and the same is true of the segments in the sequence $s_{i_{j+1}}, \ldots, s_{i_k}$. This means that $s_{i_1}\circ \cdots \circ s_{i_j}$ and $s_{i_{j+1}}\circ \cdots \circ s_{i_k}$ are well-defined acyclic paths of $T$.

The acyclic paths $s_{i_1}\circ \cdots \circ s_{i_j}$ and $s_{i_{j+1}}\circ \cdots \circ s_{i_k}$ may or may not have a common acyclic subpath. The curves $\sigma_{i_1}, \ldots, \sigma_{i_j}$ and $\sigma_{i_{j+1}}, \ldots, \sigma_{i_{k}}$ therefore appear in two possible configurations that are analogous to the two configurations in Figure~\ref{first_case_fig}. One can thus adapt the argument we have given in the case where the colors of the curves in $\varsigma(s)$ are alternating to the general case.
\end{proof}

\section{Red-green trees as torsion pairs}\label{Section_red-green_torsion_pair}

Given a facet $\mathcal{F} \in \Delta^{NC}(T)$, its red-green tree $\mathcal{T}_\mathcal{F}$ turns out to be equivalent to the data of a torsion pair inside $\text{mod}(\Lambda_T)$. This realization as a torsion pair will allow us to evaluate $\theta_\mathcal{F}$ on any indecomposable $\Lambda_T$-module. Recall that a pair of full additive subcategories $(\mathscr{T},\mathscr{F})$ of an abelian category $\mathcal{A}$ is called a \textbf{torsion pair} if the following hold:
\begin{itemize}
\item[i)] $\text{Hom}_\mathcal{A}(T, F) = 0$ for any $T \in \mathscr{T}$ and any $F \in \mathscr{F}$,
\item[ii)] for any $X \in \mathcal{A}$ there is a short exact sequence $$0 \to T \to X \to F \to 0$$ with $T \in \mathscr{T}$ and $F\in\mathscr{F}$.
\end{itemize}
The category $\mathscr{T}$ (resp., $\mathscr{F}$) is called a \textbf{torsion class} (resp., a \textbf{torsion-free class}).

Before stating our result, we recall the definition of the lattice of biclosed sets of segments of a tree as well as some of its lattice properties. We say a subset $B \subset \text{Seg}(T)$ is \textbf{closed} if for any segments $s, t \in B$ with $s\circ t \in \text{Seg}(T)$, one has that $s\circ t \in B.$ If both $B$ and $\text{Seg}(T)\backslash B$ are closed, we say $B$ is \textbf{biclosed}. We denote the poset of biclosed subsets of $\text{Seg}(T)$ ordered by inclusion by $\text{Bic}(T)$. By \cite[Theorem 4.1]{garver2016oriented}, this poset is a lattice, and the join of two biclosed sets $B_1, B_2 \in \text{Bic}(T)$ is given by $B_1 \vee B_2 = \overline{B_1\cup B_2}$. It is easy to show that any element $s \in B_1\vee B_2$ may be written as $s = t_1\circ \cdots \circ t_k$ for some $t_1, \ldots, t_k \in B_1\cup B_2$. It follows from \cite[Lemma 4.6]{garver2016oriented} that for any segment $s \in \text{Seg}(T)$ one has that $C_s, K_s \in \text{Bic}(T)$.

\begin{proposition}
Let $\textbf{B} \in \text{NCP}(T)$ be a noncrossing tree partition of a tree $T$. Then the categories
$$\mathscr{T}_\textbf{B} := \text{add}\left(\bigoplus_t M(w(t)): t \in \bigvee_{s \in \text{Seg}_g(\text{Kr}(\textbf{B}))}K_s\right), \ \ \mathscr{F}_\textbf{B}:= \text{add}\left(\bigoplus_{t} M(w(t)): t \in \bigvee_{s \in \text{Seg}_r(\textbf{B})} C_s\right)$$ define the torsion pair $(\mathscr{T}_\textbf{B},\mathscr{F}_\textbf{B})$ in $\text{mod}(\Lambda_T)$.\end{proposition}

\begin{proof}
We prove that $\text{Hom}_{\Lambda_T}(T,F)=0$ for any $T \in \mathscr{T}_\textbf{B}$ and any $F \in \mathscr{F}_\textbf{B}$. That for any $X \in \text{mod}(\Lambda_T)$ there exists the desired short exact sequence follows from Lemma~\ref{lemma_ii}.

By \cite[Lemma 8.10]{garver_oriented_rep_thy}, we know that $\text{Hom}_\Lambda(M(w(s)), M(w(t))) =0$ for any $s \in \text{Seg}_g(\text{Kr}(\textbf{B}))$ and any $t \in \text{Seg}_r(\textbf{B})$. Now given $s' \in K_s$ and $t' \in C_{t}$ where $s \in \text{Seg}_g(\text{Kr}(\textbf{B}))$ and $t \in \text{Seg}_r(\textbf{B})$, Lemma~\ref{lem2} and Remark~\ref{remark_Ks} imply that $M(w(s'))$ is a quotient of $M(w(s))$ and $M(w(t'))$ is a submodule of $M(w(t))$. We see that $\text{Hom}_\Lambda(M(w(s')),M(w(t'))) = 0,$ otherwise we have that $\text{Hom}_\Lambda(M(w(s)),M(w(t))) \neq 0$, a contradiction. 

Lastly, suppose that $s' = s'_1\circ \cdots \circ s'_k$ and $t' = t'_1\circ \cdots \circ t'_\ell$ where for each $i \in \{1,\ldots, k\}$ and each $j \in \{1,\ldots, \ell\}$ we have that $s'_i \in K_{s_i}$ for some $s_i \in \text{Seg}_g(\text{Kr}(\textbf{B}))$ and $t'_j \in C_{t_j}$ for some $t_j \in \text{Seg}_r(\textbf{B})$. One checks that there exists $i \in \{1,\ldots, k\}$ and $j \in \{1,\ldots, \ell\}$ such that $M(w(s'_i))$ is a submodule of $M(w(s'))$ and $M(w(t'_j))$ is a quotient of $M(w(t'))$. If $\text{Hom}_\Lambda(M(w(s')),M(w(t'))) \neq 0$, we obtain that $\text{Hom}_\Lambda(M(w(s'_i)), M(w(t'_j))) \neq 0$. However, this contradicts the previous paragraph. This completes the proof.
\end{proof}

\begin{lemma}\label{lemma_simples}
Given a noncrossing tree partition $\textbf{B} \in \text{NCP}(T)$, any simple module in $\text{mod}(\Lambda_T)$ belongs to $\mathscr{T}_\textbf{B}$ or $\mathscr{F}_\textbf{B}$.
\end{lemma}
\begin{proof}
Let $M(w(s)) \in \text{mod}(\Lambda_T)$ be any simple module. Here we have $s = (v_0,v_1)$. Let $\varsigma(s) = (\sigma_{i_1},\ldots, \sigma_{i_\ell})$ denote the sequence of admissible curves in $\mathcal{T}_\mathcal{F}$ joining $v_0$ and $v_1$ defined in Section~\ref{Sec_noncrossing_tree_part}. Let $(s_{i_1},\ldots, s_{i_\ell})$ denote the corresponding sequence of segments.

If $\ell = 1$, then $s = s_{i_1}$ so $s  \in \text{Seg}_g(\text{Kr}(\textbf{B}))$ or $s \in \text{Seg}_r(\textbf{B})$. This means $M(w(s)) \in \mathscr{T}_\textbf{B}$ or $M(w(s)) \in \mathscr{F}_\textbf{B}.$ Thus we can assume that $\ell > 1.$ To complete the proof, by Lemma~\ref{lem2} and Remark~\ref{remark_Ks}, it is enough to show that $M(w(s))$ is a submodule of $M(w(s_{i_j})) \in \mathscr{F}_\textbf{B}$ for some $j \in \{1,\ldots, \ell\}$ or a quotient of $M(w(s_{i_j})) \in \mathscr{T}_\textbf{B}$ for some $j \in \{1,\ldots, \ell\}.$

If $\sigma_{i_1}$ is green (resp., red) and $M(w(s))$ is a quotient (resp., a submodule) of $M(w(s_{i_1}))$, then we are done. {Suppose that neither case holds. Then, we know that $s$ is not a subsegment of $s_{i_1}$.} 

\textbf{\textit{Case 1:} the curve $\sigma_{i_1}$ is green.}  Let $j$ be the smallest element of $\{1,\ldots, \ell\}$ such that $s_{i_j}$ contains $s$ as a subsegment. Assume for a contradiction that $\sigma_{i_j}$ is green, then the curves $\sigma_{i_1}$, $\sigma_{i_{j-1}}$, and $\sigma_{i_j}$ must appear in one of the configurations in Figure~\ref{gamma_i_j} up to isotopy.

\begin{figure}
$$\begin{array}{ccccccccc}
\includegraphics[scale=1.65]{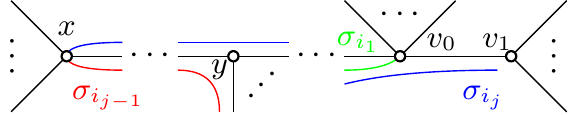}\\ (a) \\  \includegraphics[scale=1.65]{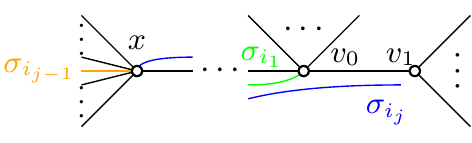} \\ (b)\\ \includegraphics[scale=1.65]{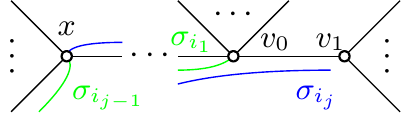} \\ (c)
\end{array}$$
\caption{{The three configurations from the proof of Lemma~\ref{lemma_simples}. In each of these configurations the curve $\sigma_{i_j}$ is a green admissible curve that is drawn in blue in order to distinguish it from $\sigma_{i_{1}}$ and $\sigma_{i_{j-1}}$. {The curve $\sigma_{i_j}$ may turn left or right at $v_1$.} In $(a)$, the vertex $y$ is a vertex of the acyclic path $[x, v_0]$ that is not equal to $x$. In $(b)$, the curve $\sigma_{i_{j-1}}$ is drawn in brown because it could be red or green.}}
\label{gamma_i_j}
\end{figure}



Observe that $[v_0,x]$ is an acyclic path of $T$ {with $v_0, x \in V^\circ_T$ and that $\varsigma([v_0,x]) = (\sigma_{i_1}, \ldots, \sigma_{i_{j-1}})$}. {If the curves $\sigma_{i_1}$, $\sigma_{i_{j-1}}$, and $\sigma_{i_j}$ belong to a configuration of the form shown in Figure~\ref{gamma_i_j} $(a)$ where $\sigma_{i_{j-1}}$ does not end at vertex $y$ or Figure~\ref{gamma_i_j} $(b)$, then the presence of $\sigma_{i_{j-1}}$ in the sequence $\varsigma([v_0,x])$ contradicts Lemma~\ref{lemma_forbidden_curves}, applied to $\varsigma([v_0,x])$.} 

{Next, suppose that the curves $\sigma_{i_1}$, $\sigma_{i_{j-1}}$, and $\sigma_{i_j}$ belong to a configuration of the form shown in Figure~\ref{gamma_i_j} $(a)$ where $\sigma_{i_{j-1}}$ ends at $y$. Since $s_{i_j}$ is the first segment in $(s_{i_1}, \ldots, s_{i_\ell})$ that contains $s$ as a subsegment, there exists a vertex $y^\prime \in V^\circ_T$ as in Figure~\ref{new_figure} $(a^\prime)$ and an index $k \in \{2,\ldots, j-2\}$ such that $\varsigma([y^\prime,x]) = (\sigma_{i_k}, \ldots, \sigma_{i_{j-1}})$. However, the presence of $\sigma_{i_{j-1}}$ in $\varsigma([y^\prime,x]) = (\sigma_{i_k},\ldots, \sigma_{i_{j-1}})$ contradicts Lemma~\ref{lemma_forbidden_curves}, applied to $\varsigma([y^\prime,x])$.} 

{Now suppose that the curves $\sigma_{i_1}$, $\sigma_{i_{j-1}}$, and $\sigma_{i_j}$ belong to a configuration of the form shown in Figure~\ref{gamma_i_j} $(c)$. Since $s_{i_j}$ is the first segment in $(s_{i_1}, \ldots, s_{i_\ell})$ that contains $s$ as a subsegment, there exists a vertex $y \in V^\circ_T$ as in Figure~\ref{new_figure} $(c^\prime)$ and an index $k \in \{2,\ldots, j-2\}$ such that $\varsigma([y,x]) = (\sigma_{i_k}, \ldots, \sigma_{i_{j-1}})$. However, the presence of $\sigma_{i_{j-1}}$ in $\varsigma([y,x]) = (\sigma_{i_k},\ldots, \sigma_{i_{j-1}})$ contradicts Lemma~\ref{lemma_forbidden_curves}, applied to $\varsigma([y,x])$.}


\begin{figure}
$$\begin{array}{ccccccccc} \includegraphics[scale=1.65]{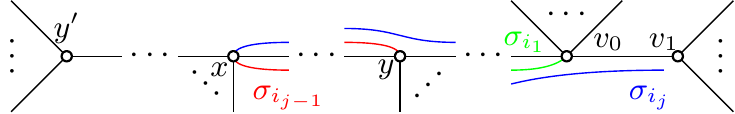} \\ (a^\prime) \\ \includegraphics[scale=1.65]{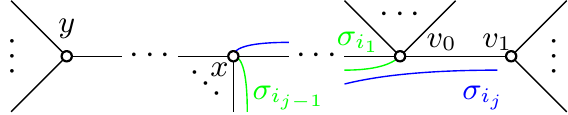} \\ (c^\prime) \end{array}$$
\caption{{In $(a^\prime)$, we show the vertex $y^\prime$ from the proof when $\sigma_{i_1}$, $\sigma_{i_{j-1}}$, and $\sigma_{i_j}$ belong to a configuration of the form shown in Figure~\ref{gamma_i_j} $(a)$ where $\sigma_{i_{j-1}}$ ends at vertex $y$. In $(c^\prime)$, we show the vertex $y$ from the proof when $\sigma_{i_1}$, $\sigma_{i_{j-1}}$, and $\sigma_{i_j}$ belong to a configuration of the form shown in Figure~\ref{gamma_i_j} $(c)$.}}
\label{new_figure}
\end{figure}


We obtain that $\sigma_{i_j}$ is red. {If $\sigma_{i_j}$ turns left at $v_1$ or if $v_1$ is an endpoint of $\sigma_{i_j}$, then} $M(w(s))$ is a submodule of $M(w(s_{i_j})) \in \mathscr{F}_\textbf{B}.$ {If $\sigma_{i_j}$ turns right at $v_1$, then it appears in the configuration in Figure~\ref{blah}. Notice that $\varsigma([x,v_1]) = (\sigma_{i_j}, \ldots, \sigma_{i_\ell})$. However, the presence of $\sigma_{i_j}$ in $\varsigma([x,v_1])$ contradicts Lemma~\ref{lemma_forbidden_curves}, applied to $\varsigma([x,v_1])$. Thus $\sigma_{i_j}$ turns left at $v_1$ or if $v_1$ is an endpoint of $\sigma_{i_j}$. Therefore, $M(w(s))$ is a submodule of $M(w(s_{i_j})) \in \mathscr{F}_\textbf{B}.$}

\begin{figure}
$$\begin{array}{ccccccccc} \includegraphics[scale=1.65]{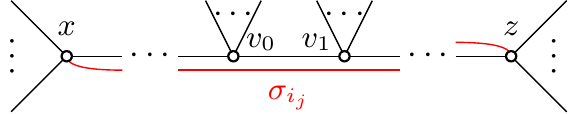} \end{array}$$
\caption{The curve $\sigma_{i_j}$ when it turns right at $v_1$.}
\label{blah}
\end{figure}

\textbf{\textit{Case 2:} the curve $\sigma_{i_1}$ is red.} As above, let $j$ be the smallest element of $\{1,\ldots, \ell\}$ such that $s_{i_j}$ contains $s$ as a subsegment. The analogous argument shows that $\sigma_{i_j}$ is green. Thus $M(w(s))$ is a quotient of $M(w(s_{i_j})) \in \mathscr{T}_\textbf{B}.$ 
\end{proof}

\begin{lemma}\label{lemma_ii}
Given a noncrossing tree partition $\textbf{B} \in \text{NCP}(T)$ and any $X \in \text{mod}(\Lambda_T)$, there exists a short exact sequence $0\to T \to X \to F \to 0$ where $T \in \mathscr{T}_\textbf{B}$ and $F \in \mathscr{F}_\textbf{B}$.
\end{lemma}

\begin{proof}
It is enough to prove the result for indecomposable $\Lambda_T$-modules. Given an indecomposable module $M(w(s)) \in \text{mod}(\Lambda_T)$, we prove the result by induction on the length of $s$.

First, suppose that $s = (v_0,v_1)$ is a minimal length segment. By Lemma~\ref{lemma_simples}, the simple module $M(w(s))$ belongs to $\mathscr{T}_\textbf{B}$ or $\mathscr{F}_\textbf{B}$. In the former case, the desired short exact sequence is given by $0 \to M(w(s)) \stackrel{1}{\to} M(w(s)) \to 0 \to 0$ where $1$ denotes the identity map on $M(w(s))$. In the latter case, the desired short exact sequence is given by $0 \to 0 \to M(w(s)) \stackrel{1}{\to} M(w(s)) \to 0$ where $1$ denotes the identity map on $M(w(s))$.

Next, suppose that $s = (v_0, \ldots, v_t) \in \text{Seg}(T)$ where $t > 1$ and that the result holds for all shorter segments. Consider the sequence $\varsigma(s) = (\sigma_{i_1}, \ldots, \sigma_{i_\ell})$, and let $(s_{i_1}, \ldots, s_{i_\ell})$ denote the corresponding sequence of segments. If $\ell = 1$, then $s = s_{i_1}$ so $M(w(s)) \in \mathscr{T}_\textbf{B}$ or $M(w(s)) \in \mathscr{F}_\textbf{B}$, by the definition of $\mathscr{T}_\textbf{B}$ and $\mathscr{F}_\textbf{B}$. In the former case, we obtain the short exact sequence $0 \to M(w(s)) \stackrel{1}{\to} M(w(s)) \to 0 \to 0$. In the latter case, we obtain the short exact sequence $0 \to 0 \to M(w(s)) \stackrel{1}{\to} M(w(s)) \to 0$. 

Now, assume that $\ell = 2$ and that the unique common endpoint of the two curves in $\varsigma(s)$ is not a vertex of $s$. This implies that $s_{i_1}$ or $s_{i_2}$ contains $s$ as a subsegment. Suppose without loss of generality that $s_{i_1}$ contains $s$ as a subsegment. Then, by Lemma~\ref{lemma_forbidden_curves}, we know that $M(w(s))$ is a quotient of $M(w(s_{i_1})) \in \mathscr{T}_\textbf{B}$ if $\sigma_{i_1}$ is green, and $M(w(s))$ is a submodule of $M(w(s_{i_1})) \in \mathscr{F}_\textbf{B}$ if $\sigma_{i_1}$ is red. We obtain that $M(w(s)) \in \mathscr{T}_\textbf{B}$ or $M(w(s)) \in \mathscr{F}_\textbf{B}$. In the former case, we obtain the short exact sequence $0 \to M(w(s)) \stackrel{1}{\to} M(w(s)) \to 0 \to 0$. In the latter case, we obtain the short exact sequence $0 \to 0 \to M(w(s)) \stackrel{1}{\to} M(w(s)) \to 0$.

Lastly, we construct the desired short exact sequence when $\ell \ge 2$ and when $\ell = 2$ we assume that two curves in the sequence $\varsigma(s)$ have a common endpoint that is a vertex of $s$ other than one of its endpoints. When $\ell = 2,$ the definition of $\varsigma(s)$ implies that if $\sigma_{i_1}$ and $\sigma_{i_2}$ have a common vertex that is a vertex of $s$, it cannot be an endpoint of $s$. Note also that when $\ell > 2$, by Lemma~\ref{lemma_varsigma}, there exists two curves in the sequence $\varsigma(s)$ that have a common endpoint that is a vertex of $s$ other than one of its endpoints. In each case, let $v_i$ with $i \neq 0, t$ denote this vertex.

Now write $s= s_1\circ s_2$ where $s_1 = (v_0, \ldots, v_i)$ and $s_2 = (v_i,\ldots, v_t)$. Without loss of generality, we assume that $M(w(s_1))$ is a submodule of $M(w(s))$ and $M(w(s_2))$ is a quotient of $M(w(s))$. By induction, there exists the following short exact sequences $$0 \longrightarrow X \stackrel{f}{\longrightarrow} M(w(s_1)) \stackrel{g}{\longrightarrow} Y \longrightarrow 0,$$
$$0 \longrightarrow X' \stackrel{f'}{\longrightarrow} M(w(s_2)) \stackrel{g'}{\longrightarrow} Y' \longrightarrow 0$$ where $X:=\bigoplus_j M(w(s_{1,j})), X':=\bigoplus_{j'} M(w(s_{2,{j'}})) \in \mathscr{T}_\textbf{B}$ and $Y:=\bigoplus_k M(w(t_{1,k})), Y':=\bigoplus_{k'} M(w(t_{2,{k'}}))\in \mathscr{F}_\textbf{B}$. No two  distinct segments belonging to the set $\{s_{1,j},s_{2,j'},t_{1,k}, t_{2,k'}\}_{j,j',k,k'}$ contain a common edge of $T$. However, there exist exactly two segments $s_1^* \in \{s_{1,j},t_{1,k}\}_{j,k}$ and $s_2^* \in \{s_{2,j'},t_{2,k'}\}_{j',k'}$ such that $s_1^*\circ s_2^* \in \text{Seg}(T)$ whose unique common vertex is $v_i$. Observe that $s_1^*$ is a subsegment of $s_1$ and $s_2^*$ is a subsegment of $s_2$.   

We construct the desired short exact sequence in each of the following four cases:
\begin{itemize}
\item[(a)] $M(w(s^*_1))$ is a submodule of $M(w(s_1)),$ $M(w(s^*_2))$ is a quotient of $M(w(s_2))$;
\item[(b)] $M(w(s^*_1))$ is a quotient of $M(w(s_1))$, $M(w(s^*_2))$ is a submodule of $M(w(s_2))$;
\item[(c)] $M(w(s^*_1))$ is a submodule of $M(w(s_1)),$ $M(w(s^*_2))$ is a submodule of $M(w(s_2))$;
\item[(d)] $M(w(s^*_1))$ is a quotient of $M(w(s_1)),$ $M(w(s^*_2))$ is a quotient of $M(w(s_2))$.
\end{itemize}

\textbf{\textit{Case (a):}} By assumption, we obtain that $M(w(s_1^*))$ is a submodule of $M(w(s))$ and $M(w(s_2^*))$ is a quotient of $M(w(s))$. Thus $$0 \longrightarrow X\oplus X' \stackrel{f\oplus f'}{\longrightarrow} M(w(s)) \stackrel{g\oplus g'}{\longrightarrow} Y\oplus Y' \longrightarrow 0$$ is a short exact sequence of $\Lambda_T$-modules with $X\oplus X' \in \mathscr{T}_\textbf{B}$ and $Y\oplus Y' \in \mathscr{F}_\textbf{B}$.

\textbf{\textit{Case (b):}} The proof is analogous to that of (a), and it produces the same short exact sequence.

\textbf{\textit{Case (c):}} By assumption, $M(w(s_1^*\circ s_2^*))$ is a submodule of $M(w(s))$. Let $i:M(w(s_1^*\circ s_2^*)) \hookrightarrow M(w(s))$ denote the canonical inclusion. Note that $M(w(s^*_1)), M(w(s^*_2)) \in \mathscr{T}_\textbf{B}$ and $s^*_1\circ s_2^* \in \bigvee_{s\in \text{Seg}_r(\textbf{B})}C_s$ since $\bigvee_{s\in \text{Seg}_r(\textbf{B})}C_s$ is closed. This implies that $M(w(s_1^*\circ s_2^*)) \in \mathscr{T}_\textbf{B}$. Thus $$0 \longrightarrow X/M(w(s_1^*))\oplus X'/M(w(s_2^*))\oplus M(w(s_1^*\circ s_2^*)) \stackrel{\overline{f}\oplus \overline{f'}\oplus i}{\longrightarrow} M(w(s)) \stackrel{g\oplus g'}{\longrightarrow} Y\oplus Y' \longrightarrow 0$$ is a short exact sequence of $\Lambda_T$-modules with $X/M(w(s_1^*))\oplus X'/M(w(s_2^*))\oplus M(w(s_1^*\circ s_2^*)) \in \mathscr{T}_\textbf{B}$ and $Y \oplus Y' \in \mathscr{F}_\textbf{B}$. Here, $\overline{f}$ (resp., $\overline{f'}$) is the induced map $\overline{f}: X/M(w(s_1^*)) \to M(w(s_1))$ (resp., $\overline{f'}:X'/M(w(s_2^*)) \to M(w(s_2))$).

\textbf{\textit{Case (d):}} The proof is analogous to that of (c). The argument produces the following short exact sequence $$0 \longrightarrow X\oplus X' \stackrel{f\oplus f'}{\longrightarrow} M(w(s)) \stackrel{\overline{g}\oplus \overline{g'} \oplus h}{\longrightarrow} Y/M(w(s_1^*))\oplus Y'/M(w(s_2^*))\oplus M(w(s_1^*\circ s_2^*))\longrightarrow 0$$ with $X\oplus X' \in \mathscr{T}_\textbf{B}$ and $Y/M(w(s_1^*))\oplus Y'/M(w(s_2^*))\oplus M(w(s_1^*\circ s_2^*)) \in \mathscr{F}_\textbf{B}$. Here $h: M(w(s)) \twoheadrightarrow M(w(s_1^*\circ s_2^*)) $ is the canonical surjection.\end{proof}

\begin{lemma}\label{lemma_varsigma}
Let $s = (v_0,\ldots, v_t) \in \text{Seg}(T)$ with $t \ge 2$, and let $\varsigma(s) = (\sigma_{i_1},\ldots, \sigma_{i_\ell})$ denote the sequence of admissible curves as defined in Section~\ref{Sec_noncrossing_tree_part}. If $\ell > 2$, there exists a vertex $v_i$ of $s$ with $i \in \{1,\ldots, t-1\}$ that is the common endpoint of two curves in $\varsigma(s)$.
\end{lemma}
\begin{proof}
Suppose no such vertex $v_i$ exists. Let $v_k$ be any vertex of $s$ where $k \in \{1,\ldots, t-1\}$. Since $\mathcal{T}_\mathcal{F}$ is a tree, there exist the two sequences $\varsigma([v_0,v_k])$ and $\varsigma([v_k,v_t])$ inside the red-green tree $\mathcal{T}_\mathcal{F}$, both of which consist of curves that do not appear in $\varsigma(s)$. However, the union of the curves appearing $\varsigma(s)$, $\varsigma([v_0,v_k]),$ and $\varsigma([v_k,v_t])$ produces a subgraph of $\mathcal{T}_{\mathcal{F}}$ that includes a cycle, a contradiction.
\end{proof}

\section{Proof of Theorem~\ref{thm:thm1}}\label{Sec_proof_of_thm}

In this section, we prove Theorem~\ref{thm:thm1}, which says that $\text{wide}(\Lambda_T) = \Lambda_T^{ss}$. In particular, it says that any wide subcategory $\mathcal{W} \in \text{wide}(\Lambda_T)$ is realizable as a semistable subcategory of $\text{mod}(\Lambda_T)$ via the stability condition $\theta_\mathcal{F}(-) = \sum_{\delta \in \mathcal{F}^{gr}}\langle \textbf{g}(\delta), -\rangle$ where $\mathcal{F}$ is the facet of $\Delta^{NC}(T)$ corresponding to $\mathcal{W}.$ We also mention the connection between our result and related work.


\begin{proof}[Proof of Theorem~\ref{thm:thm1}]
By the linearity of $\theta_{\mathcal{F}}$, it is enough to show that the indecomposable modules of $\theta^{ss}_\mathcal{F}$ are exactly the indecomposable modules of $\mathcal{W}$. Throughout the proof, we let $\textbf{B}$ denote the noncrossing tree partition corresponding to $\mathcal{F}$. 

One checks that if $\mathcal{W} = \text{mod}(\Lambda_T)$ or $\mathcal{W}$ is the zero subcategory, then the result holds. In the former case, we have $\theta_\mathcal{F}(-) = 0$. In the latter case, we have $\theta_\mathcal{F}(-) = \sum_{e \in \text{Int}(E_T)} \langle \textbf{x}_e , - \rangle$. Therefore, we can furthermore assume that $\mathcal{W} \neq \text{mod}(\Lambda_T)$ and that $\mathcal{W}$ is not the zero subcategory. This implies that $\mathcal{F}$ has at least one green arc and at least one red arc.

We first show that if $M(w(s)) \in \mathcal{W}$, then $M(w(s)) \in \theta^{ss}_\mathcal{F}.$ Assume $s \in \text{Seg}_r(\textbf{B})$.  Since $s = s_{\gamma,\mathcal{F}}$ for some red arc $\gamma \in \mathcal{F}^{red}$ and since $s$ is a red segment, it follows from Proposition~\ref{vanish} that $\theta_\mathcal{F}(M(w(s))) = 0$. 

Now, let $M(w(t))$ be a proper indecomposable submodule of $M(w(s))$. By Lemma~\ref{lem2}, we have $t \in C_{s}\backslash\{s\}.$ Any submodule of $M(w(s))$ is of the form $\bigoplus_i M(w(t_i))$ where $t_i \in C_{s}\backslash\{s\}$ for all $i$, and any distinct pair of these segments have no common subsegment. Thus it is enough to show that $\theta_\mathcal{F}(M(w(t))) < 0$.

It follows from Lemma~\ref{lem1} that we can write $\mathcal{F}^{gr} = \mathcal{F}^{gr}_1\sqcup \mathcal{F}^{gr}_2$ where
$$\begin{array}{rcl}
\mathcal{F}^{gr}_1 & = & \{\delta \in \mathcal{F}^{gr}: \ |Z^-_\delta\cap \{\text{edges of } t\}| > |Z^+_\delta\cap \{\text{edges of } t\}| \}\\
\mathcal{F}^{gr}_2 & = & \{\delta \in \mathcal{F}^{gr}: \ |Z^-_\delta\cap \{\text{edges of } t\}| = |Z^+_\delta\cap \{\text{edges of } t\}|  \}
\end{array}$$
and that the set $\mathcal{F}_1^{gr}$ is nonempty. We now have that
$$\begin{array}{rcl}
\theta_\mathcal{F}(M(w(t))) & = & \sum_{\delta \in \mathcal{F}^{gr}}\langle\textbf{g}(\delta), \textbf{dim}(M(w(t)))\rangle \\
& = & \sum_{\delta \in \mathcal{F}^{gr}} |Z^+_\delta\cap \{\text{edges of } t\}| -|Z^-_\delta\cap \{\text{edges of } t\}|\\
& = & \sum_{\delta \in \mathcal{F}_1^{gr}} |Z^+_\delta\cap \{\text{edges of } t\}| -|Z^-_\delta\cap \{\text{edges of } t\}|\\
& < & 0.
\end{array}$$
We obtain that $M(w(s)) \in \theta^{ss}_\mathcal{F}$ and that it is $\theta_\mathcal{F}$-stable.

Now assume that $s = s_1\circ \cdots \circ s_k$ where $s_i \in \text{Seg}_r(\textbf{B})$ for all $i \in \{1,\ldots, k\}$. Using Proposition~\ref{vanish}, we have that $\theta_{\mathcal{F}}(M(w(s_i))) = 0$ for all $i \in \{1,\ldots, k\}$.  By the linearity of $\theta_\mathcal{F}$, we see that $\theta_\mathcal{F}(M(w(s))) = \sum_{i =1}^k\theta_{\mathcal{F}}(M(w(s_i))) = 0.$ One checks that there exists $j \in \{1,\ldots, k\}$ such that $M(w(s_j))$ is a submodule of $M(w(s))$. Thus not all submodules $M$ of $M(w(s))$ satisfy $\theta_\mathcal{F}(M) < 0$. However, using an argument similar to that which appears in the previous two paragraphs, we have that $\theta_\mathcal{F}(M) \le 0$ for all submodules $M$ of $M(w(s))$. We obtain that $M(w(s)) \in \theta^{ss}_\mathcal{F}.$

Next, suppose $M(w(s)) \not \in \mathcal{W}.$ Observe that either $s \in \text{Seg}(T)\backslash\left(\overline{\text{Seg}_g(\text{Kr}(\textbf{B}))}\cup\overline{ \text{Seg}_r(\textbf{B})}\right)$ or $s \in \overline{\text{Seg}_g(\text{Kr}(\textbf{B}))}$. First, assume that $s \in \overline{\text{Seg}_g(\text{Kr}(\textbf{B}))}$ and $s = s_1\circ \cdots \circ s_k$ for some $s_1,\ldots, s_k \in \text{Seg}_g(\text{Kr}(\textbf{B})).$ Then Proposition~\ref{vanish} implies that $\theta_\mathcal{F}(M(w(s))) = \sum_{i = 1}^k \theta_\mathcal{F}(M(w(s_i)))  = k >0$ so $M(w(s)) \not \in \theta^{ss}_\mathcal{F}.$ 

Now assume that $s \in \text{Seg}(T)\backslash\left(\overline{\text{Seg}_g(\text{Kr}(\textbf{B}))}\cup\overline{ \text{Seg}_r(\textbf{B})}\right)$. Let $0 \to T \to M(w(s)) \to F \to 0$ denote a short exact sequence with $T \in \mathscr{T}_\textbf{B}$ and $F \in \mathscr{F}_\textbf{B}$ whose existence is guaranteed by Lemma~\ref{lemma_ii}. By assumption, we know that $T$ is a nonzero module. Let $M(w(t))$ with $t = t_1\circ \cdots \circ t_k$ for some $t_1, \ldots, t_k \in \bigvee_{s \in \text{Seg}_g(\text{Kr}(\textbf{B}))} K_s$ be a summand of $T$. We can further assume that for each $i \in \{1,\ldots, k\}$ one has $t_i \in K_{s_i}$ for some $s_i \in \text{Seg}_g(\text{Kr}(\textbf{B})).$

As mentioned earlier in the proof, one checks that there exists $j \in \{1,\ldots, k\}$ such that $M(w(t_j))$ is a submodule of $M(w(t)).$ Now let $s_j = s_{\delta,\mathcal{F}}$ where $\delta$ is a green arc of $\mathcal{F}$. Since $t_j \in K_{ s_{\delta,\mathcal{F}}}$, we have that $\langle \textbf{g}(\delta),\textbf{dim}(M(w(t_j)))\rangle > 0$. A similar argument to that which appears in the proof of Lemma~\ref{lem1} shows that for any arc $\delta \in \mathcal{F}$ and any segment $t' \in K_{s_{\gamma, \mathcal{F}}}$ one has that $|Z^-_\delta\cap \{\text{edges of $t$}\}| \le |Z^+_\delta\cap \{\text{edges of $t$}\}|.$ It follows that $\theta_{\mathcal{F}}(M(w(t_j))) > 0$. As $M(w(t_j))$ is a submodule of $M(w(s))$, we conclude that $M(w(s)) \not \in \theta_\mathcal{F}^{ss}$.

The final assertion that every semistable subcategory is a wide subcategory was proved in \cite{king1994moduli}.\end{proof}

\begin{remark}\label{ingallsthomas}
If $T$ is a tree all of whose interior vertices have degree 3 and has no subconfiguration of the form shown in Figure~\ref{bad_subtree}, then $Q_T$ is a type $A$ Dynkin quiver. Theorem~\ref{thm:thm1} therefore recovers Ingalls' and Thomas' bijection in \cite[Theorem 1.1]{ingalls2009noncrossing} between wide subcategories and semistable subcategories in type $A$. 
\end{remark}

\begin{figure}
$$\includegraphics[scale=1]{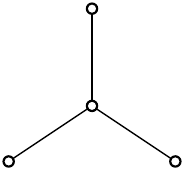}$$
\caption{The forbidden subconfigurations from Remark~\ref{ingallsthomas}. Here all four vertices are interior vertices.}
\label{bad_subtree}
\end{figure}

\begin{remark}
{The assertion in Theorem~\ref{thm:thm1} that all wide subcategories of $\text{mod}(\Lambda_T)$ are semistable can be deduced from \cite[Theorem 1.2]{yurikusa}. Our work differs from that of Yurikusa in that we have found a combinatorial construction of a  stability condition realizing a wide subcategory as a semistable subcategory.}
\end{remark}

From Theorem~\ref{thm:thm1} and \cite[Theorem 7.1]{garver_oriented_rep_thy}, it follows that the poset of semistable subcategories of mod$(\Lambda_T)$ is isomorphic to the lattice of noncrossing tree partitions of $T$ as the following corollary shows. In particular, we obtain a combinatorial classification of the semistable subcategories of mod$(\Lambda_T)$.

\begin{corollary}
For any tree $T$, the map $\rho: \text{NCP}(T)\to \Lambda^{ss}_T$ is an isomorphism of posets.
\end{corollary}

We remark that the unique minimal noncrossing tree partition $\{\{v\}: v \in V^{\circ}_T\}$ is sent to the zero subcategory of $\text{mod}(\Lambda_T)$ under this isomorphism. Similarly, the unique maximal noncrossing tree partition $\{V_T^{\circ}\}$ is sent to the category $\text{mod}(\Lambda_T)$ under this isomorphism.

\section{Additional questions}\label{Sec_additiona_stuff}

A crucial step in proving Theorem~\ref{thm:thm1} was the use of the combinatorics of the red-green tree $\mathcal{T}_\mathcal{F}$ to evaluate $\theta_\mathcal{F}$ on any indecomposable $\Lambda_T$-module. A second crucial step in the proof was the fact that for any facet $\mathcal{F}$ of the noncrossing complex, the $\textbf{g}$-vectors in $G(\mathcal{F})$ and the $\textbf{c}$-vectors in $C(\mathcal{F})$ are dual bases with respect to $\langle -, - \rangle$. This fact has already been established for general \textbf{gentle algebras} of which tiling algebras are examples (see \cite[Proposition 4.16]{palu2017non}). There the role of the noncrossing complex is played by the \textbf{blossoming complex}. However, we are not aware of a notion of the red-green tree associated to a facet of this complex. We propose the following problem.

\begin{problem}
Find a combinatorial description of the Kreweras complement $\text{Kr}:\text{wide}(\Lambda) \to \text{wide}(\Lambda)$ where $\Lambda$ is a gentle algebra. Then, given a wide subcategory $\mathcal{W}$ whose corresponding facet of the blossoming complex is $\mathcal{F}$, use this description to determine when the Kreweras stability condition $\theta_{\mathcal{F}}$ satisfies $\theta^{ss}_{\mathcal{F}} = \mathcal{W}$ for all facets $\mathcal{F}$ of the blossoming complex.
\end{problem}

\section*{Acknowledgements}{This project began at a Mitacs Globalink Research Internship at Universit\'e du Qu\'ebec \`a Montr\'eal. M. Garcia was supported by Mitacs Globalink and the project CONACyT-238754. A. Garver was supported by NSERC grant RGPIN/05999-2014 and the Canada Research Chairs Program. The authors thank an anonymous referee for careful comments that helped to improve the manuscript.}

\bibliographystyle{plain}
\bibliography{sample.bib}

\end{document}